\documentclass[twoside,11pt]{article}

\usepackage{amsmath}
\usepackage{amssymb}
\usepackage{comment}
\usepackage{graphicx}
\usepackage{cancel}
\usepackage{natbib}

\usepackage{dsfont}
\usepackage{xcolor}

\usepackage{ stmaryrd }
 
\usepackage[T1]{fontenc}
\usepackage{setspace}

\usepackage{wrapfig} 
\usepackage{times}
\usepackage{graphicx} %
\usepackage{subfigure} 
\graphicspath{{./},{./figures/}}

\usepackage{wrapfig}

\usepackage[normalem]{ulem}

\newcommand{\Bm}{\mathrm{B}}
\newcommand{\Lm}{\mathrm{L}^\alpha}

\newcommand{\sas}{\mathcal{S}\alpha\mathcal{S}}

\usepackage{dsfont}
\usepackage{soul}

\numberwithin{equation}{section}

\newcommand\eqdist{\stackrel{\mathclap{\tiny\mbox{$\mathrm{d}$}}}{=}}

\graphicspath{{./figures/}{./}}

\newcommand{\rmd}{\mathrm{d}}

\usepackage[preprint]{jmlr2e}

\usepackage{cleveref}

\Crefname{assumption}{\textbf{A}\hspace{-3pt}}{\textbf{A}\hspace{-3pt}}
\crefname{assumption}{\textbf{A}}{\textbf{A}}

\jmlrheading{1}{2020}{1-26}{02/20}{--/--}{tbd}{G\"urbuzbalaban et al.}

\ShortHeadings{The Heavy-Tail Phenomenon in SGD}{G\"urb\"uzbalaban et al.\ }  
\firstpageno{1}

\usepackage{microtype}
\usepackage{graphicx}
\usepackage{subfigure}
\usepackage{booktabs} %

\usepackage{hyperref}
\usepackage{soul}

\usepackage{amsmath,amsfonts,amssymb}
\usepackage{enumerate,color,xcolor}
\usepackage{graphicx}
\usepackage{subfigure}
\usepackage{url}

\usepackage{enumitem}

\usepackage{multirow}

\newcommand{\beq}{\begin{eqnarray}}
\newcommand{\eeq}{\end{eqnarray}}
\newcommand{\beqs}{\begin{eqnarray*}}
\newcommand{\eeqs}{\end{eqnarray*}}

\usepackage{mathtools}

\usepackage{silence}
\ActivateWarningFilters[pdftoc]
\ActivateWarningFilters[xclr]

\begin{document}

\title{The Heavy-Tail Phenomenon in SGD}

\author{\name Mert G\"{u}rb\"{u}zbalaban$^\text{1}$ \email mg1366@rutgers.edu 
\AND
\name Umut \c{S}im\c{s}ekli$^\text{2}$ \email umut.simsekli@inria.fr 
       \AND
       \name Lingjiong Zhu$^\text{3}$ \email zhu@math.fsu.edu\\
      \addr 1: Department of Management Science and Information Systems, Rutgers Business School, Piscataway, USA \\
      2: INRIA - D\'{e}pt.\ d'Informatique de l'\'{E}cole Normale Sup\'{e}rieure - PSL Research University, Paris, France \\
      3: Department of Mathematics, Florida State University, Tallahassee, USA \\
      }

\editor{TBD}

\maketitle

\begin{abstract}
In recent years, various notions of capacity and complexity have been proposed for characterizing the generalization properties of stochastic gradient descent (SGD) in deep learning. Some of the popular notions that correlate well with the performance on unseen data are (i) the `flatness' of the local minimum found by SGD, which is related to the eigenvalues of the Hessian, (ii) the ratio of the stepsize $\eta$ to the batch-size $b$, which essentially controls the magnitude of the stochastic gradient noise, and (iii) the `tail-index', which measures the heaviness of the tails of the network weights at convergence. In this paper, we argue that these three seemingly unrelated perspectives for generalization are deeply linked to each other. We claim that depending on the structure of the Hessian of the loss at the minimum, and the choices of the algorithm parameters $\eta$ and $b$, {the distribution of} the SGD iterates will converge to a \emph{heavy-tailed} stationary distribution. We rigorously prove this claim in the setting of quadratic optimization: we show that even in a simple linear regression problem with independent and identically distributed data whose distribution has finite moments of all order, the iterates can be heavy-tailed with infinite variance. We further characterize the behavior of the tails with respect to algorithm parameters, the dimension, and the curvature. We then translate our results into insights about the behavior of SGD in deep learning. We support our theory with experiments conducted on synthetic data, fully connected, and convolutional neural networks.
\end{abstract}

\section{Introduction}

The learning problem in neural networks can be expressed as an instance of the well-known \emph{population risk minimization} problem in statistics, given as follows: 
\begin{equation}
\min\nolimits_{x \in \mathbb{R}^d} F(x):= \mathbb{E}_{z \sim \mathcal{D}}[ f(x, z)],
\label{pbm-pop-min}
\end{equation}
where $z \in \mathbb{R}^p$ denotes a random data point, $\mathcal{D}$ is a probability distribution on $\mathbb{R}^p$ that denotes the law of the data points, $x \in \mathbb{R}^d$ denotes the parameters of the neural network to be optimized, and $f : \mathbb{R}^d \times \mathbb{R}^p \mapsto \mathbb{R}_+ $ denotes a measurable cost function, which is often non-convex in $x$. While this problem cannot be attacked directly since $\mathcal{D}$ is typically unknown, if we have access to a \emph{training dataset} $S = \{z_1, \dots, z_n\}$ with $n$ independent and identically distributed (i.i.d.) observations, i.e., $z_i \sim_{\text{i.i.d.}} \mathcal{D}$ for $i=1,\dots, n$, we can use the \emph{empirical risk minimization} strategy, which aims at solving the following optimization problem \citep{shalev2014understanding}: 
\begin{align} 
\min_{x \in \mathbb{R}^d}  f(x) := f(x,S) := (1/n) \sum\nolimits_{i=1}^n f^{(i)}(x),
\label{pbm-emp-risk}
\end{align} 
where $f^{(i)}$ denotes the cost induced by the data point $z_i$. The stochastic gradient descent (SGD) algorithm has been one of the most popular algorithms for addressing this problem:
\begin{align}\label{eqn:sgd}
&x_k = x_{k-1} - \eta \nabla \tilde{f}_{k} (x_{k-1}),
\\
&\text{ where } \>\>\>\> \nabla \tilde{f}_k(x) := (1/b) \sum\nolimits_{i\in \Omega_k} \nabla f^{(i)}(x). 
\nonumber
\end{align}
Here, $k$ denotes the iterations, $\eta > 0$ is the stepsize (also called the learning-rate), $\nabla \tilde{f}$ is the stochastic gradient, $b$ is the batch-size, and $\Omega_k \subset \{1,\dots,n\}$ is a random subset with $|\Omega_k|=b$ for all $k$.


%

Even though the practical success of SGD has been proven in many domains, the theory for its generalization properties is still in an early phase. Among others, one peculiar property of SGD that has not been theoretically well-grounded is that, depending on the choice of $\eta$ and $b$, the algorithm can exhibit significantly different behaviors in terms of the performance on unseen test data. 

A common perspective over this phenomenon  
%
%
is based on the `flat minima' argument that dates back to \citet{hochreiter1997flat}, and associates the performance with the `sharpness' or `flatness' of the minimizers found by SGD, where  
these notions are often characterized by the magnitude of the eigenvalues of the Hessian, larger values corresponding to sharper local minima \citep{keskar2016large}. 
%
%
Recently, \citet{jastrzkebski2017three} focused on this phenomenon as well and empirically illustrated that the performance of SGD on unseen test data is mainly determined by the stepsize $\eta$ and the batch-size $b$, i.e., larger $\eta/b$ yields better generalization. Revisiting the flat-minima argument, they concluded that the ratio $\eta/b$ determines the flatness of the minima found by SGD; hence the difference in generalization.
%
%
In the same context, \citet{pmlr-v97-simsekli19a} focused on the statistical properties of the gradient noise $(\nabla \tilde{f}_k (x) - \nabla f(x))$ and illustrated that under an isotropic model, the gradient noise exhibits a heavy-tailed behavior, which was also confirmed in follow-up studies \citep{zhang2019adam,zhou2020towards}. Based on this observation and a metastability argument \citep{Pavlyukevich_2007}, they showed that SGD will `prefer' wider basins under the heavy-tailed noise assumption, without an explicit mention of the cause of the heavy-tailed behavior. More recently, \citet{xie2021a} studied SGD with anisotropic noise and showed with a density diffusion theory approach that it favors flat minima.


In another recent study, \citet{martin2019traditional} introduced a new approach for investigating the generalization properties of deep neural networks by invoking results from heavy-tailed random matrix theory. They empirically showed that the eigenvalues of the weight matrices in different layers exhibit a \emph{heavy-tailed} behavior, which is an indication that the weight matrices themselves exhibit heavy tails as well \citep{arous2008spectrum}. 
Accordingly, they fitted a power law distribution to the empirical spectral density of individual layers and illustrated that 
heavier-tailed weight matrices indicate better generalization. Very recently, \citet{csimcsekli2020hausdorff} formalized this argument in a mathematically rigorous framework and showed that such a heavy-tailed behavior diminishes the `effective dimension' of the problem, which in turn results in improved generalization. While these studies form an important initial step towards establishing the connection between heavy tails and generalization, the \emph{originating cause} of the observed heavy-tailed behavior is yet to be understood.

\textbf{Contributions.} In this paper,  we argue that these three seemingly unrelated perspectives for generalization are deeply linked to each other. We claim that, depending on the choice of the algorithm parameters $\eta$ and $b$, the dimension $d$, and the curvature of $f$ (to be precised in Section~\ref{sec:main}), SGD exhibits a `heavy-tail phenomenon', meaning that the law of the iterates converges to a heavy-tailed distribution. We rigorously prove that, this phenomenon is not specific to deep learning and in fact it can be observed even in surprisingly simple settings: we show that when $f$ is chosen as a simple quadratic function and the data points are i.i.d.\ from a 
continuous distribution supported on $\mathbb{R}^d$ with light tails, 
{the distribution of} the iterates can still converge to a heavy-tailed distribution with arbitrarily heavy tails, hence with infinite variance. 
%
If in addition, the input data is isotropic Gaussian, we are able to provide a sharp characterization of the tails where we show that 
(i) the tails become \emph{monotonically heavier} for increasing curvature, increasing $\eta$, or decreasing $b$, hence relating the heavy-tails to the ratio $\eta/b$ and the curvature,
(ii) the law of the iterates converges exponentially fast towards the stationary distribution in the Wasserstein metric, 
(iii) there exists a higher-order moment (e.g., variance) of the iterates that diverges \emph{at most} polynomially-fast, depending on the heaviness of the tails at stationarity. More generally, if the input data is not Gaussian, our monotonicity results extend where we can show that a lower bound on the thickness of the tails (which will be defined formally in Section \ref{sec:main}) is monotonic with respect to $\eta, b,d$ and the curvature. To the best of our knowledge, these results are the first of their kind to rigorously characterize the empirically observed heavy-tailed behavior of SGD with respect to the parameters $\eta$, $b$, $d$, and the curvature, with explicit convergence rates.\footnote{We note that in a concurrent work, which very recently appeared on arXiv, \citet{hodgkinson2020multiplicative} showed that heavy tails with power laws arise in more general Lipschitz stochastic optimization algorithms that are contracting on average for strongly convex objectives near infinity with positive
probability. Our Theorem~\ref{thm:main} and Lemma~\ref{rho:h:iid} are more refined as we focus on the special case of SGD for linear regression, where we are able to provide constants which \emph{explicitly} determine the tail-index as an expectation over data and SGD parameters (see also eqn. (\ref{eq-rho-hs})). Due to the generality of their framework, Theorem~1 in \citet{hodgkinson2020multiplicative} is more implicit and it cannot provide such a characterization of these constants, however it can be applied to other algorithms beyond SGD. All our other results (including Theorem~\ref{thm:mono} -- monotonicity of the tail-index and Corollary~\ref{cor:clt} -- central limit theorem for the ergodic averages) are all specific to SGD and cannot be obtained under the framework of \citet{hodgkinson2020multiplicative}. We encourage the readers to refer to \citet{hodgkinson2020multiplicative} for the treatment of more general stochastic recursions.}
Finally, we support our theory with experiments conducted on both synthetic data and neural networks. Our experimental results provide strong empirical support that our theory extends to deep learning settings for both fully connected and convolutional networks. \looseness=-1
\section{Technical Background}
\label{sec:bg}

\textbf{Heavy-tailed distributions with a power-law decay.} 
A real-valued random variable $X$ is said to be \emph{heavy-tailed}
if the right tail or the left tail of the distribution  decays slower than any exponential distribution. 
We say $X$ has heavy (right) tail
if $\lim_{x\rightarrow\infty}\mathbb{P}(X\geq x)e^{cx}=\infty$
for any $c>0$.
\footnote{A real-valued random variable $X$ has heavy (left) tail 
if $\lim_{x\rightarrow\infty}\mathbb{P}(X\leq-x)e^{c|x|}=\infty$
for any $c>0$.} Similarly, an $\mathbb{R}^{d}$-valued random vector $X$
has heavy tail if $u^{T}X$ has heavy right tail 
for some vector $u\in\mathbb{S}^{d-1}$, 
where $\mathbb{S}^{d-1}:=\{u\in\mathbb{R}^{d}:\Vert u\Vert=1\}$ is the unit sphere in $\mathbb{R}^{d}$.

Heavy tail distributions include $\alpha$-stable distributions,
Pareto distribution, log-normal distribution and the Weilbull distribution. One important class of the heavy-tailed distributions
is the distributions with \emph{power-law} decay, 
which is the focus of our paper. That is,
$\mathbb{P}(X\geq x)\sim c_{0}x^{-\alpha}$
as $x\rightarrow\infty$ for some $c_{0}>0$ and $\alpha>0$,
where $\alpha>0$ is known as the \emph{tail-index},
which determines the tail thickness of the distribution.
Similarly, we say that the random vector $X$ has power-law decay
with tail-index $\alpha$ if for some $u\in\mathbb{S}^{d-1}$,
we have $\mathbb{P}(u^{T}X\geq x)\sim c_{0}x^{-\alpha}$,
for some $c_{0},\alpha>0$.

\textbf{Stable distributions.} 
The class of $\alpha$-stable distributions 
are an important subclass of heavy-tailed distributions with a power-law decay,
which appears as the limiting distribution of the generalized CLT for
a sum of i.i.d. random variables
with infinite variance \citep{paul1937theorie}.
A random variable $X$ follows a symmetric $\alpha$-stable 
distribution denoted as $X\sim\mathcal{S}\alpha\mathcal{S}(\sigma)$ 
if its characteristic function
takes the form:
\begin{equation*}
\mathbb{E}\left[e^{it X}\right]=\exp\left(-\sigma^{\alpha}|t|^{\alpha}\right), \qquad t\in\mathbb{R},
\end{equation*}
where $\sigma>0$ is the scale parameter that measures the spread
of $X$ around $0$,
and $\alpha\in(0,2]$ is known as the tail-index,
and $\mathcal{S}\alpha\mathcal{S}$ becomes
heavier-tailed as $\alpha$ gets smaller.
The probability density function of a symmetric $\alpha$-stable distribution, $\alpha\in(0,2]$,
does not yield closed-form expression in general except for a few special cases.
When $\alpha=1$ and $\alpha=2$, $\sas$ reduces to the Cauchy and the Gaussian distributions, respectively.
When $0<\alpha<2$, $\alpha$-stable distributions 
have their moments
being finite only up to the order $\alpha$
in the sense that 
$\mathbb{E}[|X|^{p}]<\infty$
if and only if $p<\alpha$,
which implies infinite variance. 

\textbf{Wasserstein metric.}
For any $p\geq 1$, define $\mathcal{P}_{p}(\mathbb{R}^{d})$
as the space consisting of all the Borel probability measures $\nu$
on $\mathbb{R}^{d}$ with the finite $p$-th moment
(based on the Euclidean norm).
For any two Borel probability measures $\nu_{1},\nu_{2}\in\mathcal{P}_{p}(\mathbb{R}^{d})$, 
we define the standard $p$-Wasserstein
metric \citep{villani2008optimal}:
$$\mathcal{W}_{p}(\nu_{1},\nu_{2}):=\left(\inf\mathbb{E}\left[\Vert Z_{1}-Z_{2}\Vert^{p}\right]\right)^{1/p},$$
where the infimum is taken over all joint distributions of the random variables $Z_{1},Z_{2}$ with marginal distributions
$\nu_{1},\nu_{2}$.


\section{Setup and Main Theoretical Results}\label{sec:main}
We first observe that SGD  \eqref{eqn:sgd} is an iterated random recursion of the form 
 $ x_k = \Psi(x_{k-1},\Omega_k)$, 
where the map $\Psi:\mathbb{R}^d\times \mathcal{S} \to \mathbb{R}^d$, $\mathcal{S}$ denotes the set of all subsets of $\{1,2,\dots,N\}$ and $\Omega_k$ is random and i.i.d. If we write $\Psi_{\Omega}(x) = \Psi(x,\Omega)$ for notational convenience where $\Omega$ has the same distribution as $\Omega_k$, then $\Psi_{\Omega}$ is a random map and 
  \begin{equation} x_k = \Psi_{\Omega_k}(x_{k-1}).
   \label{eq-random-IFS} 
   \end{equation} 
Such random recursions are studied in the literature. If this map is Lipschitz on average, i.e.  
\begin{equation} \mathbb{E}[L_\Omega] < \infty, \, \mbox{with} \, L_\Omega := \sup\nolimits_{x,y\in\mathbb{R}^d} \frac{\|\Psi_{\Omega}(x) - \Psi_{\Omega}(y) \|}{\|x-y\|},
\label{stat-sol}
\end{equation}
and is mean-contractive, i.e. if $\mathbb{E}\log(L_\Omega)<0$ then it can be shown under further technical assumptions that the distribution of the iterates converges to a unique stationary distribution $x_\infty$ geometrically fast (the Prokhorov distance is proportional to $\rho^k$ for some $\rho<1$) although the rate of convergence $\rho$ is not explicitly known in general \citep{diaconis1999iterated}. However, much less is known about the tail behavior of the limiting distribution $x_\infty$ except when the map $\Psi_\Omega(x)$ has a linear growth for large $x$.
The following result characterizes the tail-index under such assumptions for dimension $d=1$. We refer the readers to \citet{mirek2011heavy} for general $d$.
\begin{theorem}(\citet{mirek2011heavy}, see also \citet{buraczewski2016stochastic})\label{thm-mirek} Assume stationary solution to \eqref{eq-random-IFS} exists and:

(i) {There exists a random matrix $M(\Omega)$ and a random variable $B(\Omega)>0$ such that for a.e. $\Omega$, 
$|\Psi_\Omega(x)-M(\Omega) x | \leq B(\Omega)$ for every $x$;}

(ii) The conditional law of $\log |M(\Omega)|$ given $M(\Omega)\neq 0$ is non-arithmetic;

(iii) There exists $\alpha>0$ such that $\mathbb{E} |M(\Omega)|^\alpha = 1$,  $\mathbb{E} |B(\Omega)|^\alpha <\infty$ and $\mathbb{E} [|M(\Omega)|^\alpha \log^+ |M(\Omega)|]< \infty$ where $\log^+(x) := \max(\log(x),0)$.

Then, it holds that $\lim_{t \to \infty } t^{\alpha} \mathbb{P}(|x_\infty|>t) = c_0$ for some constant $c_0>0$.  
\end{theorem}
Relaxations of the assumptions of Theorem~\ref{thm-mirek} which require only lower and upper bounds on the growth of $\Psi_\Omega$ have also been recently developed \citep{hodgkinson2020multiplicative,alsmeyer2016stationary}. {Unfortunately, 
it is highly non-trivial to verify such assumptions in practice, and
furthermore, the literature does not provide any rigorous connections between the tail-index $\alpha$ and the choice of the stepsize, batch-size in SGD or the curvature of the objective at hand which is key to relate the tail-index to the generalization properties of SGD.}

Before stating our theoretical results in detail, let us informally motivate our main method of analysis. Suppose the initial SGD iterate $x_0$ is in the domain of attraction\footnote{We say $x_0$ is in the domain of attraction of a local minimum $x_\star$, if gradient descent iterations to minimize $f$ started at $x_0$ with sufficiently small stepsize converge to $x_\star$ as the number of iterations goes to infinity.} of a local minimum $x_\star$ of $f$ which is smooth and well-approximated by a quadratic function in this basin. Under this assumption, by considering a first-order Taylor approximation of $\nabla f^{(i)}(x)$ around $x_\star$, we have 
\begin{align*}
\nabla f^{(i)}(x) \approx \nabla f^{(i)}(x_\star) + \nabla^2 f^{(i)}(x_\star) (x - x_\star). 
\end{align*}
By using this approximation, we can approximate the SGD recursion (\ref{eqn:sgd}) as:
\begin{align}
\nonumber x_k &\approx  x_{k-1} - (\eta/b) \sum\nolimits_{i \in \Omega_{k}}  \nabla^2 f^{(i)}(x_\star) x_{k-1} 
\nonumber
\\
&\quad+ (\eta/b) \sum\nolimits_{i \in \Omega_k} \left( \nabla^2 f^{(i)}(x_\star) x_\star -\nabla f^{(i)}(x_\star) \right)
\nonumber
\\
&=:   \left(I - (\eta/b)H_k\right) x_{k-1} + q_k, \label{eq-stoc-recur}
\end{align}
where $I$ denotes the identity matrix of appropriate size. Here, our main observation is that the SGD recursion can be approximated by an \emph{affine stochastic recursion}. In this case, the map $\Psi_\Omega(x)$ is affine in $x$, and in addition to Theorem~\ref{thm-mirek}, we have access to the tools from Lyapunov stability theory \cite{srikant2019finite} and implicit renewal theory for investigating its statistical properties \citep{kesten1973random,goldie1991implicit}. In particular, \citet{srikant2019finite} study affine stochastic recursions subject to Markovian noise with a Lyapunov approach and show that the lower-order moments of the iterates can be made small as a function of the stepsize while they can be upper-bounded by the moments of a Gaussian
random variable. In addition, they provide some examples where higher-order moments are infinite in steady-state. In the renewal theoretic approach, the object of interest would be the matrix $(I - \frac{\eta}{b}H_k)$ which determines the behavior of $x_k$: depending on the moments of this matrix, $x_k$ can have heavy or light tails, or might even diverge. \looseness=-1

In this study, we focus on the tail behavior of the SGD dynamics by analyzing it through the lens of implicit renewal theory. As, the recursion (\ref{eq-stoc-recur}) is obtained by a quadratic approximation of the component functions $f^{(i)}$, which arises naturally in linear regression, we will consider a simplified setting and study it in great depth this dynamics in the case of linear regression. As opposed to prior work, this formalization will enable us to derive sharp characterizations of the tail-index and its dependency to the parameters $\eta, b$ and the curvature as well as rate of convergence $\rho$ to the stationary distribution. Our analysis technique
lays the first steps for 
the analysis of more general objectives, and our experiments provide strong empirical support that our theory extends to deep learning settings. 

We now focus on the case when $f$ is a quadratic, which arises in linear regression: 
\begin{equation}
\min\nolimits_{x\in\mathbb{R}^{d}} F(x) := (1/2)\mathbb{E}_{{(a,y)}\sim  \mathcal{D}} \left[ \left(a^T x - y\right)^2 \right]\,,
\label{pbm-lse}
\end{equation}
where the data $(a,y)$ comes from an unknown distribution $\mathcal{D}$ with support $\mathbb{R}^d \times\mathbb{R}$. Assume we have access to i.i.d.\ samples $(a_i,y_i)$ from the distribution $\mathcal{D}$ where
$ \nabla f^{(i)}(x) = a_i(a_i^T x - y_i)$
is an unbiased estimator of the true gradient $\nabla F(x)$. The curvature, i.e. the value of second partial derivatives, of this objective around a minimum is determined by the Hessian matrix $\mathbb{E}(aa^T)$ which depends on the distribution of $a$. 
In this setting, SGD with batch-size $b$ leads to the iterations 
%
\begin{align}\label{eq-stoc-grad-2}
&x_{k} = M_{k} x_{k-1}  + q_{k} \text{ with} \,\, M_{k}:=I -(\eta/b) H_{k}, 
\\
&H_k := \sum\nolimits_{i\in \Omega_k } a_i a_i^T, \, q_k := (\eta/b) \sum\nolimits_{i\in \Omega_k } a_i y_i\,,
\nonumber
\end{align}
where $\Omega_k := \{b(k-1)+1, b(k-1)+2, \dots, bk\}$ with $|\Omega_k| = b$. Here, for simplicity, we assume that we are in the one-pass regime (also called the streaming setting \citep{frostig2015competing,jain2017accelerating,gao2018global}) where each sample is used only once without being recycled. Our purpose in this paper is to show that heavy tails can arise in SGD even in simple settings such as when the input data $a_i$ is Gaussian, \emph{without the necessity to have a heavy-tailed input data}\footnote{Note that if the input data is heavy-tailed, the stationary distribution of SGD automatically becomes heavy-tailed; see \citet{buraczewski2012asymptotics} for details. In our context, the challenge is to identify the occurrence of the heavy tails when the distribution of the input data is light-tailed, such as a simple Gaussian distribution.}. Consequently, we make the following assumptions on the data throughout the paper: 
\begin{itemize}
    \item [\textbf{(A1)}] 
    $a_i$'s are i.i.d. with a continuous distribution supported on $\mathbb{R}^d$ with all the moments finite. All the moments of $a_i$ are finite.
    \item [\textbf{(A2)}] $y_i$ are i.i.d. with a continuous density whose support is $\mathbb{R}$ with all the moments finite. 
\end{itemize}
We assume $\textbf{(A1)}$ and $\textbf{(A2)}$ throughout the paper, 
and they are satisfied in a large variety of cases, for instance when $a_{i}$ and $y_i$ are normally distributed. 
Let us introduce
\begin{equation}
h(s) := \lim\nolimits_{k\to\infty}\left(\mathbb{E}\| M_k M_{k-1}\dots M_1\|^s\right)^{1/k}\,,
\label{def-hs}
\end{equation}
which arises in stochastic matrix recursions (see e.g. \citet{buraczewski2014multidimensional}) where $\|\cdot\|$ denotes the matrix 2-norm (i.e. largest singular value of a matrix).  
Since $\mathbb{E}\|M_k\|^s < \infty$ for all $k$ and $s>0$, we have $h(s) < \infty$.  Let us also define $\Pi_k := M_k M_{k-1}\dots M_1$ and
\begin{equation}
\rho := \lim\nolimits_{k\to\infty} (2k)^{-1} \log\left(\mbox{largest eigenvalue of } \Pi_k^T \Pi_k\right)\,.
\label{def-rho}
\end{equation}
The latter quantity is 
called the top Lyapunov exponent of the stochastic recursion (\ref{eq-stoc-grad-2}). 
Furthermore, if $\rho$ exists and is negative, it can be shown that
a stationary distribution of the recursion (\ref{eq-stoc-grad-2}) exists. 

Note that by Assumption \textbf{(A1)}, the matrices $M_k = I - \frac{\eta}{b} H_k$ are i.i.d. and 
by Assumption \textbf{(A3)}, the Hessian matrix of the objective (\ref{pbm-lse}) satisfies $\mathbb{E}(aa^T)=\sigma^2 I_d$ where the value of $\sigma^2$ determines the \emph{curvature} around a minimum; smaller (larger) $\sigma^2$ implies the objective will grow slower (faster) around the minimum and the minimum will be flatter (sharper) (see e.g. \citet{dinh2017sharp}).

In the following, we show that the limit density
has a polynomial tail
with a tail-index given precisely by $\alpha$, the unique critical value
such that $h(\alpha)=1$. 
The result builds on adapting the techniques developed in stochastic matrix recursions \citep{alsmeyer2012tail,buraczewski2016stochastic} to our setting. Our result shows that even in the simplest setting when the input data is i.i.d.\ without any heavy tail, SGD iterates can lead to a heavy-tailed stationary distribution with an infinite variance. To our knowledge, this is the first time such a phenomenon is proven in the linear regression setting. \looseness=-1


\begin{theorem}\label{thm:main}
Consider the SGD iterations (\ref{eq-stoc-grad-2}).
If $\rho<0$ and there exists a unique positive $\alpha$ such that $h(\alpha)=1$, {then 
\eqref{eq-stoc-grad-2} admits a unique stationary solution $x_{\infty}$
and the SGD iterations converge to $x_{\infty}$ in distribution,
where the distribution of $x_{\infty}$ satisfies}
\begin{equation} 
\lim\nolimits_{t\to\infty} t^\alpha \mathbb{P}\left(u^T x_{\infty} > t \right)= e_\alpha(u)\,, \quad u\in\mathbb{S}^{d-1}\,,\label{eq-heavy-tail}
\end{equation}
for some positive and continuous function $e_\alpha$ on 
$\mathbb{S}^{d-1}$.
\end{theorem}



{The proofs of Theorem~\ref{thm:main} and all the following results in the main paper are given in the Appendix.} As \citet{martin2019traditional,csimcsekli2020hausdorff} provide numerical and theoretical evidence showing that the tail-index $\alpha$ of the density of the network weights is closely related to the generalization performance, where smaller $\alpha$ indicates better generalization,
a natural question of practical importance is \emph{how the tail-index depends on the parameters of the problem including the batch-size, dimension and the stepsize}. 
In order to have a more explicit characterization of the tail-index, 
we will make the following additional assumption for the rest of the paper which says that the input is Gaussian.\looseness=-1 
 
\begin{itemize}
    \item [\textbf{(A3)}] $a_{i}\sim\mathcal{N}(0,\sigma^{2}I_{d})$ are Gaussian distributed for every $i$.
\end{itemize}
Under \textbf{(A3)}, next result shows that the formulas for $\rho$ and $h(s)$ can be simplified. 
Let $H$ be a matrix with the same distribution as $H_k$, 
and $e_{1}$ be the first basis vector. Define
\begin{align}
&\tilde{\rho}:=\mathbb{E}\log\left\Vert\left(I-(\eta/b) H\right)e_{1}\right\Vert, \nonumber
\\
&\tilde{h}(s):=\mathbb{E}\left[\left\Vert\left(I-(\eta/b)H\right)e_{1}\right\Vert^{s}\right] \mbox{ for } \rho<0.\label{eq-rho-hs}
\end{align}

\begin{theorem}\label{thm:Gaussian}
{Assume \textbf{(A3)} holds.}
Consider the SGD iterations (\ref{eq-stoc-grad-2}).
If $\rho<0$, then the following holds:

(i) There exists a unique positive $\alpha$ such that $h(\alpha)=1$
and \eqref{eq-heavy-tail} holds;

(ii) We have $\rho=\tilde{\rho}$ and $h(s)=\tilde{h}(s)$, where $\tilde{\rho}$
and $\tilde{h}(s)$ are defined in \eqref{eq-rho-hs}.
\end{theorem}

This connection will allow us to get finer characterizations of the stepsize and batch-size choices that will provably lead to heavy tails with an infinite variance.\looseness=-1

When input is not Gaussian (Theorem~\ref{thm:main}), the explicit formula (\ref{eq-rho-hs}) for $\rho$ and $h(s)$ will not hold as an equality but it will become the following inequality:
\begin{align}
&\rho\leq\hat{\rho}:=\mathbb{E}\log\left\Vert\left(I-(\eta/b) H\right)\right\Vert,\nonumber
\\
    &h(s) \leq \hat{h}(s):=\mathbb{E}\left[\left\Vert\left(I-(\eta/b)H\right)\right\Vert^{s}\right], 
    \label{ineq-h-s}
\end{align}
where $\rho$ and $h(s)$ are defined by (\ref{def-hs}). This inequality is just a consequence of sub-multiplicativity of the norm of matrix products appearing in (\ref{def-hs}). If $\hat{\alpha}$ is such that $\hat{h}(\hat\alpha)=1$, then by (\ref{ineq-h-s}), $\hat{\alpha}$ is a lower bound on the tail-index $\alpha$ that satisfies $h(\alpha)=1$ where $h$ is defined as in (\ref{def-hs}). In other words, when the input is not Gaussian, we have $\hat{\alpha}\leq \alpha$ and therefore $\hat{\alpha}$ serves as a lower bound on the tail-index.
Finally, we remark that $\hat{\rho}$ and $\hat{h}(s)$ can help us check the conditions in Theorem~\ref{thm:main}.
Since $\rho\leq\hat{\rho}$, we have $\rho<0$ when
$\hat{\rho}<0$. Moreover, $h(0)=\hat{h}(0)=1$, and one can check that $h(s)$ is convex in $s$. When $\hat{\rho}<0$, $\hat{h}'(0)=\hat{\rho}<0$, and $h(s)\leq\hat{h}(s)<1$ for any sufficiently small $s>0$. Under some mild assumption {on the data distribution}, one can check that $\liminf_{s\rightarrow\infty}h(s)>1$ and thus there exists a unique positive $\alpha$ such that $h(\alpha)=1$.

When input is Gaussian satisfying \textbf{(A3)}, due to the spherical symmetry of the Gaussian distribution, we have also $\rho=\hat{\rho}$, $h(s)=\hat{h}(s)$ (see Lemma \eqref{rho:h:iid}). 
Furthermore, in this case, 
by using the explicit characterization of the tail-index $\alpha$ in Theorem~\ref{thm:Gaussian},
we prove that larger batch-sizes lead to a lighter tail (i.e. larger $\alpha$), which links the heavy tails to the observation that smaller $b$ yields improved generalization in a variety of settings in deep learning \citep{keskar2016large,panigrahi2019non,martin2019traditional}.
We also prove that smaller stepsizes lead to larger $\alpha$, hence lighter tails, which agrees with the fact that the existing literature for linear regression often choose $\eta$ small enough to guarantee that variance of the iterates stay bounded \citep{dieuleveut2017harder,jain2017accelerating}.\looseness=-1

\begin{theorem}\label{thm:mono}  
{Assume \textbf{(A3)} holds.} The tail-index $\alpha$ is strictly increasing in batch-size $b$
and strictly decreasing in stepsize $\eta$ and variance $\sigma^{2}$ provided that $\alpha\geq 1$.
Moreover, the tail-index $\alpha$ is strictly decreasing in dimension $d$.\looseness=-1
\end{theorem}


When input is not Gaussian, Theorem~\ref{thm:mono} can be adapted in the sense that $\hat\alpha$ (defined via $\hat{h}(\hat{\alpha})=1$) will be strictly increasing in batch-size $b$ and strictly increasing in stepsize $\eta$ and variance $\sigma^2$ provided that $\hat{\alpha}\geq 1$.

Under \textbf{(A3)}, next result characterizes the tail-index $\alpha$ depending on the choice of the batch-size $b$, the variance $\sigma^2$, which determines the curvature around the minimum and the stepsize; in particular we show that if the stepsize exceeds an explicit threshold, the stationary distribution will become heavy tailed with an infinite variance.

\begin{proposition}\label{alpha:2}
Assume \textbf{(A3)} holds. Let $\eta_{crit} = \frac{2b}{\sigma^{2}(d+b+1)}.$
The following holds: 

(i) There exists $\eta_{max}>\eta_{crit}$ such that
for any $\eta_{crit}<\eta<\eta_{max}$, Theorem~\ref{thm:main} holds with tail-index $0<\alpha< 2$. 

(ii) If $\eta = \eta_{crit}$, Theorem~\ref{thm:main} holds with tail-index $\alpha=2$. 

(iii) If $\eta \in (0,\eta_{crit})$, then Theorem~\ref{thm:main} holds with tail-index $\alpha>2$.
\end{proposition}
 
\textbf{Relation to first exit times.} Proposition~\ref{alpha:2} implies that, for fixed $\eta$ and $b$, the tail-index $\alpha$ will be decreasing with increasing $\sigma$. Combined with the first-exit-time analyses of \citet{pmlr-v97-simsekli19a,nguyen2019first}, which state that the escape probability from a basin becomes higher for smaller $\alpha$, our result implies that the probability of SGD escaping from a basin gets larger with increasing curvature; hence providing an alternative view for the argument that SGD prefers flat minima.\looseness=-1

\textbf{Three regimes for stepsize.} Theorems~\ref{thm:main}-\ref{thm:mono} and Proposition~\ref{alpha:2} identify three regimes: (I) convergence to a limit with a finite variance if $\rho<0$ and $\alpha>2$; (II) convergence to a heavy-tailed limit with infinite variance if $\rho<0$ and $\alpha<2$; (III) $\rho>0$ when convergence cannot be guaranteed. For Gaussian input, if the stepsize is small enough, smaller than $\eta_{crit}$, by Proposition~\ref{alpha:2}, $\rho<0$ and $\alpha>2$, therefore regime (I) applies. As we increase the stepsize, there is a critical stepsize level $\eta_{crit}$ for which $\eta >\eta_{crit}$ leads to $\alpha<2$ as long as $\eta <\eta_{max}$ where $\eta_{max}$ is the maximum allowed stepsize for ensuring convergence (corresponds to $\rho=0$).
A similar behavior with three (learning rate) stepsize regimes was reported in \citet{lewkowycz2020large} and derived analytically for one hidden layer linear networks with a large width. The large stepsize choices that avoids divergence, so called the \emph{catapult phase} for the stepsize, yielded the best generalization performance empirically, driving the iterates to a flatter minima in practice. We suspect that the catapult phase in \citet{lewkowycz2020large} corresponds to regime (II) in our case, where the iterates are heavy-tailed, which might cause convergence to flatter minima as the first-exit-time discussions suggest \citep{csimcsekli2019heavy}. 

\textbf{Moment Bounds and Convergence Speed. }
Theorem~\ref{thm:main} is of asymptotic nature which characterizes the stationary distribution $x_{\infty}$ of
SGD iterations 
with a tail-index $\alpha$. 
Next, we provide non-asymptotic moment
bounds for $x_{k}$ at each $k$-th iterate,
and also for the limit $x_{\infty}$.

\begin{theorem}\label{thm:moment}
Assume \textbf{(A3)} holds.

(i) If the tail-index $\alpha\leq 1$, 
then for any $p\in(0,\alpha)$, we have $h(p)<1$ and
\begin{equation*}
\mathbb{E}\Vert x_{k}\Vert^{p}
\leq
(h(p))^{k}\mathbb{E}\Vert x_{0}\Vert^{p}
+\frac{1-(h(p))^{k}}{1-h(p)}\mathbb{E}\Vert q_{1}\Vert^{p}.
\end{equation*}

(ii) If the tail-index $\alpha>1$, 
then for any $p\in(1,\alpha)$, we have $h(p)<1$
and for any $0<\epsilon<\frac{1}{h(p)}-1$, we have
\begin{equation*}
\mathbb{E}\Vert x_{k}\Vert^{p}
\leq
((1+\epsilon)h(p))^{k}\mathbb{E}\Vert x_{0}\Vert^{p}
+\frac{1-((1+\epsilon)h(p))^{k}}{1-(1+\epsilon)h(p)}
\frac{(1+\epsilon)^{\frac{p}{p-1}}-(1+\epsilon)}{((1+\epsilon)^{\frac{1}{p-1}}-1)^{p}}
\mathbb{E}\Vert q_{1}\Vert^{p}.
\end{equation*}
\end{theorem}

Theorem~\ref{thm:moment} shows that when $p<\alpha$ the upper bound on the $p$-th moment of the iterates converges exponentially to the $p$-the moment of $q_1$ when $\alpha\leq 1$ and a neighborhood of the $p$-moment of $q_1$ when $\alpha>1$, where $q_{1}$ is defined in (\ref{eq-stoc-grad-2}). By letting $k\rightarrow\infty$
and applying Fatou's lemma, we can also characterize the moments of the stationary distribution. 

\begin{corollary}\label{cor:moment}
Assume \textbf{(A3)} holds.

(i) If the tail-index $\alpha\leq 1$, 
then for any $p\in(0,\alpha)$, 
\begin{equation*}
\mathbb{E}\Vert x_{\infty}\Vert^{p}
\leq
\frac{1}{1-h(p)}\mathbb{E}\Vert q_{1}\Vert^{p}, 
\end{equation*}
where $h(p)<1$.

(ii) If the tail-index $\alpha>1$, 
then for any $p\in(1,\alpha)$, we have $h(p)<1$
and for any $\epsilon>0$ such that $(1+\epsilon)h(p)<1$, we have
\begin{equation*}
\mathbb{E}\Vert x_{\infty}\Vert^{p}
\leq
\frac{1}{1-(1+\epsilon)h(p)}
\frac{(1+\epsilon)^{\frac{p}{p-1}}-(1+\epsilon)}{((1+\epsilon)^{\frac{1}{p-1}}-1)^{p}}
\mathbb{E}\Vert q_{1}\Vert^{p}.
\end{equation*}
\end{corollary}

Next, we will study the speed of convergence of the $k$-th iterate $x_{k}$ to its stationary distribution $x_{\infty}$ in the Wasserstein metric $\mathcal{W}_{p}$ for any $1\leq p<\alpha$. 

\begin{theorem}\label{thm:conv}
Assume \textbf{(A3)} holds.
Assume $\alpha>1$.
Let $\nu_{k}$, $\nu_{\infty}$ denote the probability laws
of $x_{k}$ and $x_{\infty}$ respectively. 
Then  
\begin{equation*}
\mathcal{W}_{p}(\nu_{k},\nu_{\infty})
\leq
(h(p))^{k/p}
\mathcal{W}_{p}(\nu_{0},\nu_{\infty}),
\end{equation*}
for any $1\leq p<\alpha$,
where the convergence rate $(h(p))^{1/p}\in(0,1)$.\looseness=-1
\end{theorem}
Theorem~\ref{thm:conv} shows that in case $\alpha<2$ the convergence to a heavy tailed distribution occurs relatively fast, i.e. with a linear convergence in the $p$-Wasserstein metric. 
We can also characterize the constant $h(p)$ in Theorem~\ref{thm:conv} which controls the convergence rate as follows:
\begin{corollary}\label{cor:conv}
Assume \textbf{(A3)} holds.
When $\eta<\eta_{crit}=\frac{2b}{\sigma^{2}(d+b+1)}$,
we have the tail-index $\alpha>2$, and 
\begin{align*}
\mathcal{W}_{2}(\nu_{k},\nu_{\infty})
\leq
\left(1-2\eta\sigma^{2}
\left(1-\eta/\eta_{crit}\right)\right)^{k/2}
\mathcal{W}_{2}(\nu_{0},\nu_{\infty}).
\end{align*}
\end{corollary}

Theorem~\ref{thm:conv} works for any $p<\alpha$.
At the critical $p=\alpha$,
Theorem~\ref{thm:main} indicates that
$\mathbb{E}\Vert x_{\infty}\Vert^{\alpha}=\infty$,
and therefore we have $\mathbb{E}\Vert x_{k}\Vert^{\alpha}\to \infty$ as $k\to \infty$,
\footnote{Otherwise, one can construct a subsequence $x_{n_k}$ that is bounded in the space $L^\alpha$ converging to $x_\infty$ which would be a contradiction.}  which serves
as an evidence that the tail gets heavier as 
the number of iterates $k$ increases. 
By adapting the proof of Theorem~\ref{thm:moment}, 
we have the following result stating that the moments of the iterates of order $\alpha$ go to infinity but this speed can only be polynomially fast. \looseness=-1
\begin{proposition}\label{prop:k:bound}
Assume \textbf{(A3)} holds.
Given the tail-index $\alpha$, 
we have $\mathbb{E}\Vert x_{\infty}\Vert^{\alpha}=\infty$.
Moreover, $\mathbb{E}\Vert x_{k}\Vert^{\alpha}=O(k)$
if $\alpha\leq 1$, and
$\mathbb{E}\Vert x_{k}\Vert^{\alpha}=O(k^{\alpha})$
if $\alpha>1$.
\end{proposition}
It may be possible to leverage recent results on the concentration of products of i.i.d. random matrices \citep{huang2020matrix,henriksen2020concentration} to study the tail of $x_{k}$ for finite $k$,
which can be a future research direction.

\textbf{Generalized Central Limit Theorem for Ergodic Averages.}
%
When $\alpha>2$, by Corollary~\ref{cor:moment}, second moment of the iterates $x_k$ are finite, in which case central limit theorem (CLT) says that if the cumulative sum of the iterates $S_{K}=\sum_{k=1}^{K} x_{k}$ 
is scaled properly, the resulting distribution is Gaussian. In the case where $\alpha<2$, the variance of the iterates is not finite; however in this case, we derive the following generalized CLT (GCLT) which says if the iterates are properly scaled, the limit will be an $\alpha$-stable  distribution. This is stated in a more precise manner as follows.
\begin{corollary}\label{cor:clt}

Assume the conditions of Theorem \ref{thm:main} are satisfied, i.e. assume 
$\rho<0$ {and there exists a unique positive $\alpha$ such that $h(\alpha)=1$}. Then, we have the following:

(i) If  $\alpha \in(0,1) \cup(1,2)$, then there is a sequence $d_{K}=d_{K}(\alpha)$  and a function $C_{\alpha}: \mathbb{S}^{d-1} \mapsto \mathbb{C}$  such that as $K\rightarrow\infty$ the random variables $K^{-\frac{1}{\alpha}}\left(S_{K}-d_{K}\right)$ converge in law to the $\alpha$-stable random variable with characteristic function
$\Upsilon_{\alpha}(tv)=\exp(t^{\alpha}C_{\alpha}(v))$, for $t>0$ and $v \in \mathbb{S}^{d-1}$.

(ii) If $\alpha=1$, then there are functions $\xi, \tau:(0, \infty) \mapsto \mathbb{R}$ and $C_{1}: \mathbb{S}^{d-1} \mapsto \mathbb{C}$ such that as $K\rightarrow\infty$ the random variables $K^{-1} S_{K}-K \xi\left(K^{-1}\right)$ converge in law to the random variable with characteristic function
$\Upsilon_{1}(t v)=\exp \left(t C_{1}(v)+i t\langle v, \tau(t)\rangle\right)$, 
for $t>0$ and $v \in \mathbb{S}^{d-1}$.

(iii) If \(\alpha=2,\) then there is a sequence \(d_{K}=d_{K}(2)\) and a function \(C_{2}: \mathbb{S}^{d-1} \mapsto \mathbb{R}\) such that as $K\rightarrow\infty$
the random variables \((K \log K)^{-\frac{1}{2}}\left(S_{K}-d_{K}\right)\) converge in law to the random variable with
characteristic function
$\Upsilon_{2}(t v)=\exp \left(t^{2} C_{2}(v)\right)$, for $t>0$ and $v \in \mathbb{S}^{d-1}$.

(iv) If \(\alpha \in(0,1),\) then \(d_{K}=0,\) and if \(\alpha \in(1,2],\) then \(d_{K}=K \bar{x},\) where \(\bar{x}=\int_{\mathbb{R}^{d}} x \nu_\infty(d x) .\) 
\end{corollary}

{In addition to its evident theoretical interest, Corollary~\ref{cor:clt} has also an important practical implication: estimating the tail-index of a \emph{generic} heavy-tailed distribution is a challenging problem (see e.g. \citet{clauset2009power,goldstein2004problems,bauke-power-law}); however, for the specific case of $\alpha$-stable distributions, accurate and computationally efficient estimators, which \emph{do not} require the knowledge of the functions $C_\alpha$, $\tau$, $\xi$, have been proposed \citep{mohammadi2015estimating}. Thanks to Corollary~\ref{cor:clt}, we will be able to use such estimators in our numerical experiments in Section~\ref{sec:exps}}. 

{\textbf{Further Discussions.}
Even though we focus on the case when $f$ is quadratic and consider the linear regression \eqref{pbm-lse},
for fully non-convex Lipschitz losses, it is possible to show that a stationary distribution exists and the distribution of the iterates converge to it exponentially fast but heavy-tails in this setting is not understood in general except some very special cases; see e.g. \citet{diaconis1999iterated}. However, if the gradients have asymptotic linear growth, even for non-convex objectives, extending our tail index results beyond quadratic optimization is possible if we incorporate the proof techniques of \citet{alsmeyer2016stationary} to our setting. However, in this more general case, characterizing the tail index explicitly and studying its dependence on stepsize, batch-size does not seem to be a tractable problem since the dependence of the asymptotic linear growth of the random iteration on the data may not be tractable, therefore studying the quadratic case allows us a deeper understanding of the tail index on a relatively simpler problem.}

We finally note that the gradient noise in SGD is actually both multiplicative and additive \citep{dieuleveut2017harder,dieuleveut2017bridging}; a fact that is often discarded for simplifying the mathematical analysis. In the linear regression setting, we have shown that the multiplicative noise $M_k$ is the main source of heavy-tails, where a deterministic $M_k$ would not lead to heavy tails.\footnote{E.g., if $M_{k}$ is deterministic and $q_{k}$ is Gaussian, then $x_{k}$ is Gaussian for all $k$, and so is $x_{\infty}$ if the limit exists.} In light of our theory, in Section~\ref{sec-appendix-sde} in the Appendix, we discuss in detail the recently proposed stochastic differential equation (SDE) representations of SGD in continuous-time and argue that, compared to classical SDEs driven by a Brownian motion \citep{jastrzkebski2017three,cheng2019stochastic}, SDEs driven by heavy-tailed $\alpha$-stable L\'{e}vy processes \citep{pmlr-v97-simsekli19a} are more adequate when $\alpha <2$.


\section{Experiments}
\label{sec:exps}

In this section, we present our experimental results on both synthetic and real data, in order to illustrate that our theory also holds in finite-sum problems (besides the streaming setting). Our main goal will be to illustrate the tail behavior of SGD by varying the algorithm parameters: depending on the choice of the stepsize $\eta$ and the batch-size $b$, {the distribution of the iterates does} converge to a heavy-tailed distribution (Theorem~\ref{thm:main}) and the behavior of the tail-index obeys Theorem~\ref{thm:mono}. Our implementations can be found in \url{github.com/umutsimsekli/sgd_ht}.



\textbf{Synthetic experiments.} 
In our first setting, we consider a simple synthetical setup, where we assume that the data points follow a Gaussian distribution. We will illustrate that the SGD iterates can become heavy-tailed even in this simplistic setting where the problem is a simple linear regression with all the variables being Gaussian. More precisely, we will consider the following model:
$x_0 \sim \mathcal{N}(0, \sigma_x^2 I)$, 
$a_i \sim \mathcal{N}(0, \sigma^2 I)$, 
and
$y_i | a_i, x_0 \sim \mathcal{N}\left(a_i^\top x_0, \sigma_y^2\right)$, 
where $x_0, a_i \in \mathbb{R}^d$, $y_i \in \mathbb{R}$ for $i=1,\dots,n$, and $\sigma, \sigma_x, \sigma_y >0$.

%
In our experiments, we will need to estimate the tail-index $\alpha$ of the stationary distribution $\nu_\infty$. Even though several tail-index estimators have been proposed for generic heavy-tailed distributions in the literature \citep{paulauskas2011once}, we observed that, even for small $d$, these estimators can yield inaccurate estimations and require tuning hyper-parameters, which is non-trivial. We circumvent this issue thanks to the GCLT in Corollary~\ref{cor:clt}: since the average of the iterates is guaranteed to converge to a multivariate $\alpha$-stable random variable {in distribution}, we can use the tail-index estimators that are specifically designed for stable distributions. By following \citet{tzagkarakis2018compressive,pmlr-v97-simsekli19a}, we use the estimator proposed by \citet{mohammadi2015estimating}, which is fortunately agnostic to the scaling function $C_\alpha$. The details of this estimator are given in Section~\ref{sec:alpha_estim} in the Appendix.  
\begin{figure}[t]
    \centering
    \subfigure[]{\label{fig:exp} 
    \includegraphics[width=0.39\columnwidth]{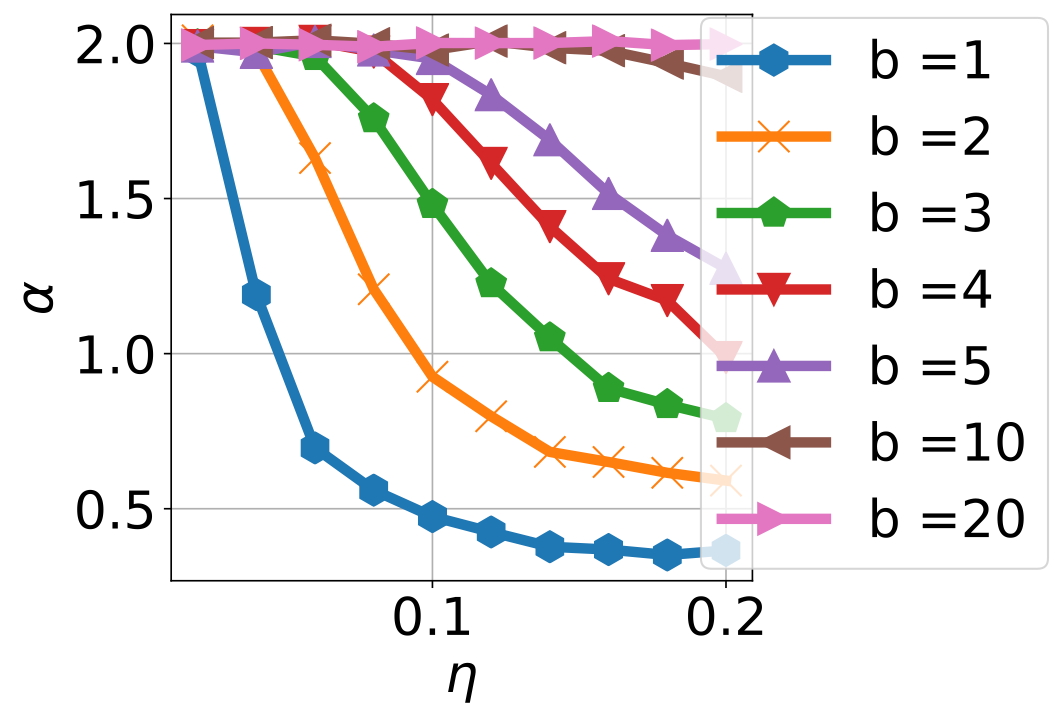}
    \includegraphics[width=0.28\columnwidth]{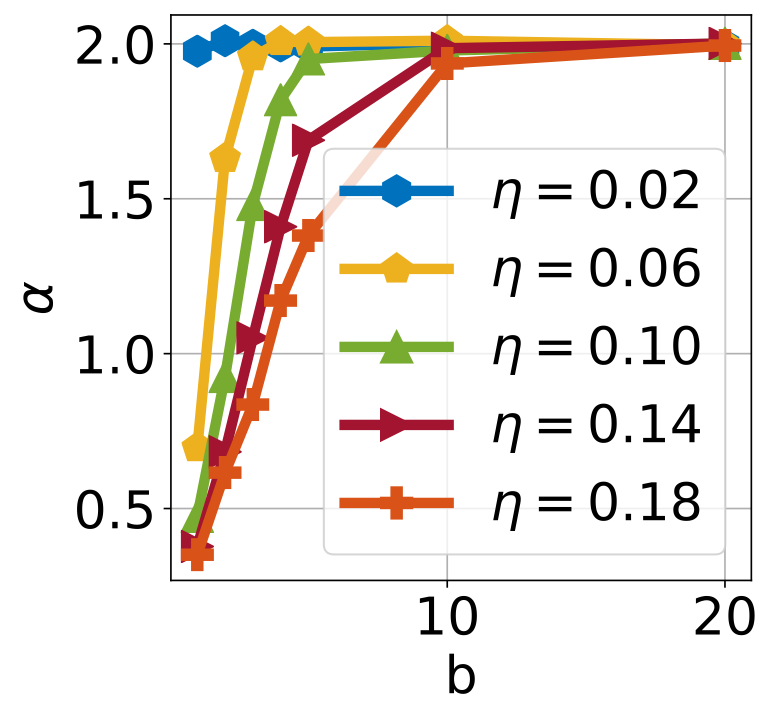}
    \includegraphics[width=0.28\columnwidth]{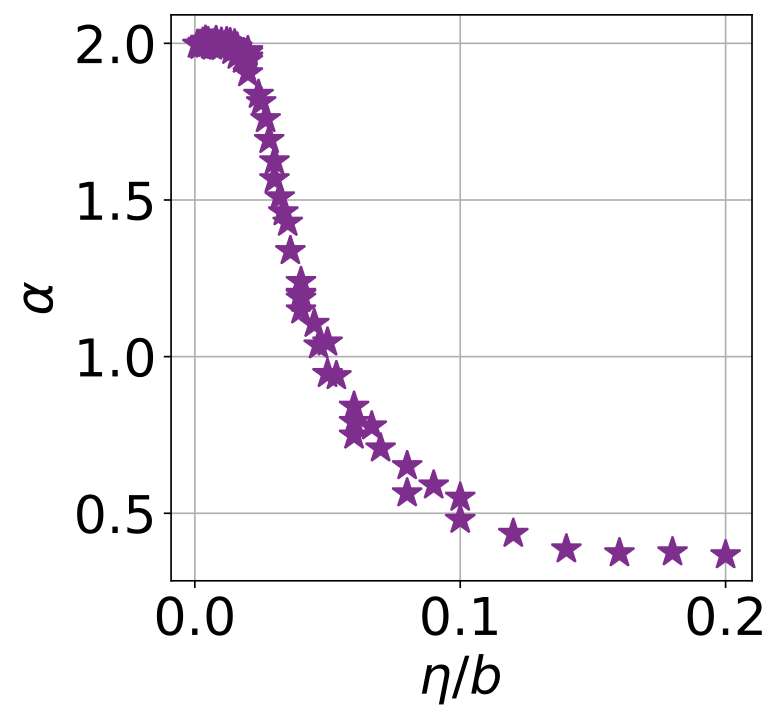}} \\ 
    \subfigure[]{
    \label{fig:exp2}
    \includegraphics[width=0.28\columnwidth]{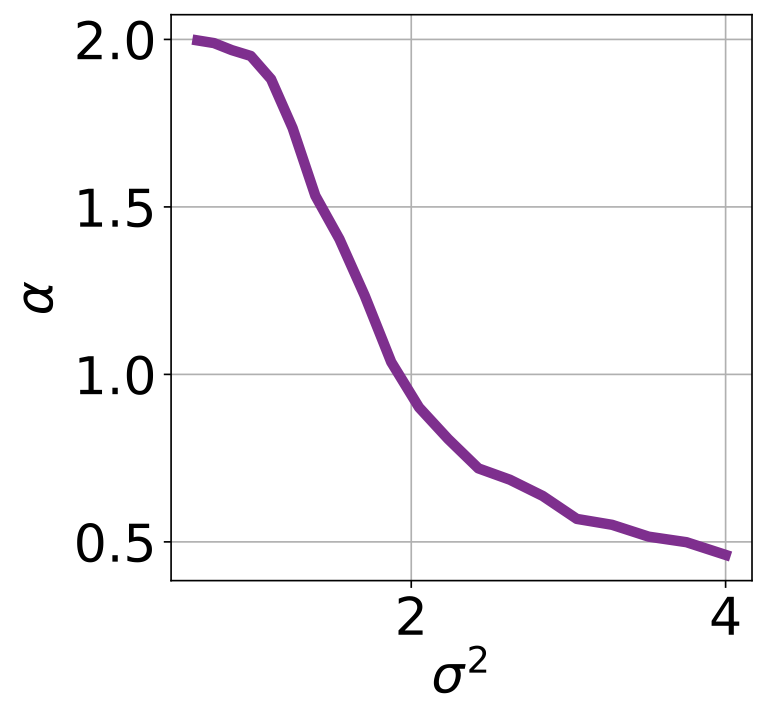}
    \includegraphics[width=0.28\columnwidth]{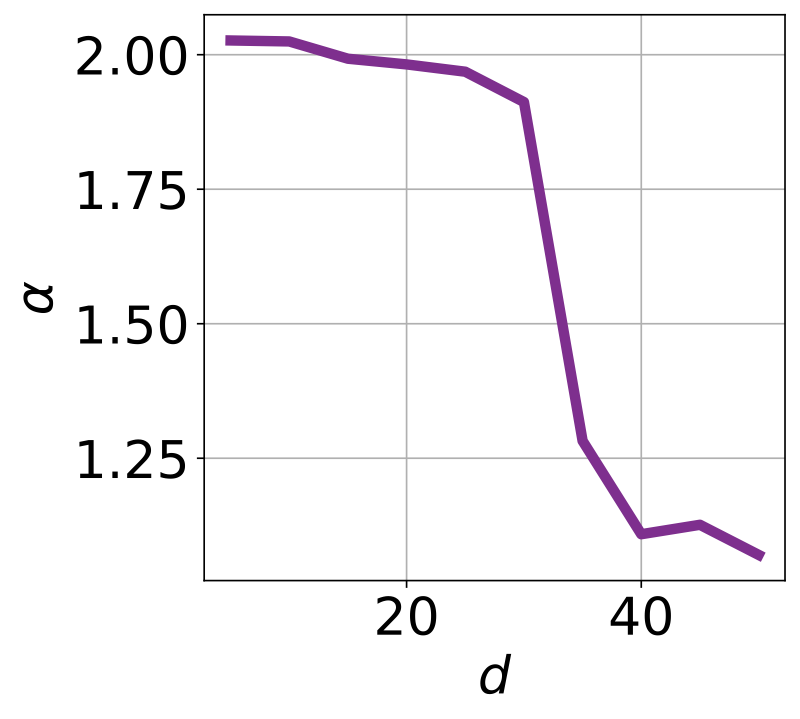}
    \hfill
    }
    \subfigure[]{
    \label{fig:exp3}
        \includegraphics[width=0.28\columnwidth]{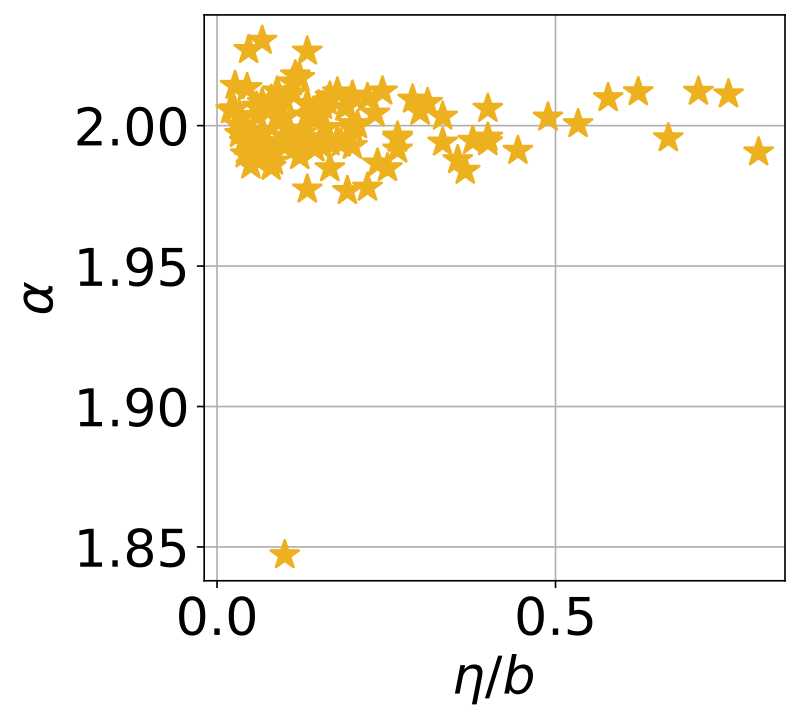}
    }

    \vspace{-10pt}
    \caption{Behavior of $\alpha$ with (a) varying stepsize $\eta$ and batch-size $b$, (b) $d$ and $\sigma$, (c) under RMSProp.}
\end{figure}

To 
benefit from the GCLT, we are required to compute the average of the `centered' iterates: 
$\frac1{K-K_0}\sum\nolimits_{k=K-K_0+1}^K (x_k - \bar{x})$, 
where $K_0$ is a `burn-in' period aiming to discard the initial phase of SGD, and the mean of $\nu_\infty$ is given by $\bar{x} = \int_{\mathbb{R}^{d}} x \nu_\infty(d x) = (A^\top A)^{-1}A^\top y$ as long as $\alpha >1$\footnote{The form of $\bar{x}$ can be verified by noticing that $\mathbb{E}[x_{k}]$ converges to the minimizer of the problem by the law of total expectation. Besides, our GCLT requires the sum of the iterates to be normalized by $\frac1{(K-K_0)^{1/\alpha}}$; however, for a finite $K$, normalizing by $\frac1{K-K_0}$ results in a scale difference, to which our estimator is agnostic. }, where the $i$-th row of $A \in \mathbb{R}^{n \times d}$ contains $a_i^\top$ and $y = [y_1, \dots, y_n] \in \mathbb{R}^n$. We then repeat this procedure $1600$ times for different initial points and obtain $1600$ different random vectors, whose distributions are supposedly close to an $\alpha$-stable distribution. Finally, we run the tail-index estimator of \citet{mohammadi2015estimating} on these random vectors to estimate $\alpha$.



In our first experiment, we investigate the tail-index $\alpha$ of the stationary measure $\nu_\infty$ for varying stepsize $\eta$ and batch-size $b$. We set $d=100$ first fix the variances $\sigma=1$, $\sigma_x=\sigma_y=3$, and generate $\{a_i, y_i\}_{i=1}^n$ by simulating the statistical model. Then, by fixing this dataset, we run the SGD recursion (\ref{eq-stoc-grad-2}) for a large number of iterations and vary $\eta$ from $0.02$ to $0.2$ and $b$ from $1$ to $20$. We also set $K=1000$ and $K_0=500$. Figure~\ref{fig:exp} illustrates the results. We can observe that, increasing $\eta$ and decreasing $b$ both result in decreasing $\alpha$, where the tail-index can be prohibitively small (i.e., $\alpha<1$, hence even the mean of $\nu_\infty$ is not defined) for large $\eta$. Besides, we can also observe that the tail-index is in strong correlation with the ratio $\eta/b$.

In our second experiment, we investigate the effect of $d$ and $\sigma$ on $\alpha$. In Figure~\ref{fig:exp2} (left), we set $d=100$, $\eta=0.1$ and $b=5$ and vary $\sigma$ from $0.8$ to $2$. For each value of $\sigma$, we simulate a new dataset from by using the generative model and run SGD with $K, K_0$. We again repeat each experiment $1600$ times. We follow a similar route for Figure~\ref{fig:exp2} (right): we fix $\sigma=1.75$ and repeat the previous procedure for each value of $d$ ranging from $5$ to $50$. The results confirm our theory: $\alpha$ decreases for increasing $\sigma$ and $d$, and we observe that for a fixed $b$ and $\eta$ the change in $d$ can abruptly alter $\alpha$. 

In our final synthetic data experiment, we investigate how the tails behave under adaptive optimization algorithms. We replicate the setting of our first experiment, with the only difference that we replace SGD with RMSProp \citep{hinton2012neural}. As shown in Figure~\ref{fig:exp3}, the `clipping' effect of RMSProp as reported in \citet{zhang2019adam,zhou2020towards} prevents the iterates become heavy-tailed and the vast majority of the estimated tail-indices is around $2$, indicating a Gaussian behavior. On the other hand, we repeated the same experiment with the variance-reduced optimization algorithm SVRG \citep{johnson2013accelerating}, and observed that for almost all choices of $\eta$ and $b$ the algorithm converges near the minimizer (with an error in the order of $10^{-6}$), hence the stationary distribution $\nu_\infty$ seems to be a degenerate distribution, which does not admit a heavy-tailed behavior. Regarding the link between heavy-tails and generalization \citep{martin2019traditional,csimcsekli2020hausdorff}, this behavior of RMSProp and SVRG might be related to their ineffective generalization as reported in \citet{keskar2017improving,defazio_2019}.

\textbf{Experiments on fully connected neural networks.}
In the second set of experiments, we investigate the applicability of our theory beyond the quadratic optimization problems. Here, we follow the setup of \citet{csimcsekli2019heavy} and consider a fully connected neural network with the cross entropy loss and ReLU activation functions on the MNIST and CIFAR10 datasets. 
We train the models by using SGD for $10$K iterations and we range $\eta$ from $10^{-4}$ to $10^{-1}$ and $b$ from $1$ to $10$. Since it would be computationally infeasible to repeat each run thousands of times as we did in the synthetic data experiments, in this setting we follow a different approach based on (i) \citep{csimcsekli2019heavy} that suggests that the tail behavior can differ in different layers of a neural network, and (ii) \citep{de2020quantitative} that shows that in the infinite width limit, the different components of a given layer of a two-layer fully connected network (FCN) becomes independent. Accordingly, we first compute the average of the last $1$K SGD iterates, whose distribution should be close an $\alpha$-stable distribution by the GCLT. We then treat each layer as a collection of i.i.d.\ $\alpha$-stable random variables and measure the tail-index of each individual layer separately by using the the estimator from \citet{mohammadi2015estimating}.
Figure~\ref{fig:exp_nn} shows the results for a three-layer network (with $128$ hidden units at each layer) , whereas we obtained very similar results with a two-layer network as well. We observe that, while the dependence of $\alpha$ on $\eta/b$ differs from layer to layer, in each layer the measured $\alpha$ correlate very-well with the ratio $\eta/b$ in both datasets.



\textbf{Experiments on VGG networks.} 
In our last set of experiments, we evaluate our theory on VGG networks \citep{simonyan2014very} on CIFAR10 with $11$ layers ($10$ convolutional layers with max-pooling and ReLU units, followed by a final linear layer), which contains $10$M parameters. We follow the same procedure as we used for the fully connected networks, where we vary $\eta$ from $10^{-4}$ to $1.7\times 10^{-3}$ and $b$ from $1$ to $10$. 
The results are shown in Figure~\ref{fig:vgg}. Similar to the previous experiments, we observe that $\alpha$ depends on the layers. For the layers 2-8, the tail-index correlates well with the ratio $\eta/b$, whereas the first and layers 1, 9, and 10 exhibit a Gaussian behavior ($\alpha \approx 2$). On the other hand, the correlation between the tail-index of the last layer (which is linear) with $\eta/b$ is still visible, yet less clear. Finally, in the last plot, we compute the median of the estimate tail-indices over layers, and observe a very clear decrease with increasing $\eta/b$. These observations provide further support for our theory and show that the heavy-tail phenomenon also occurs in neural networks, whereas $\alpha$ is potentially related to $\eta$ and $b$ in a more complicated way. 

\begin{figure}[t!]
    \setlength{\abovecaptionskip}{-3pt}
    \subfigure[MNIST]{
    \includegraphics[width=0.315\columnwidth]{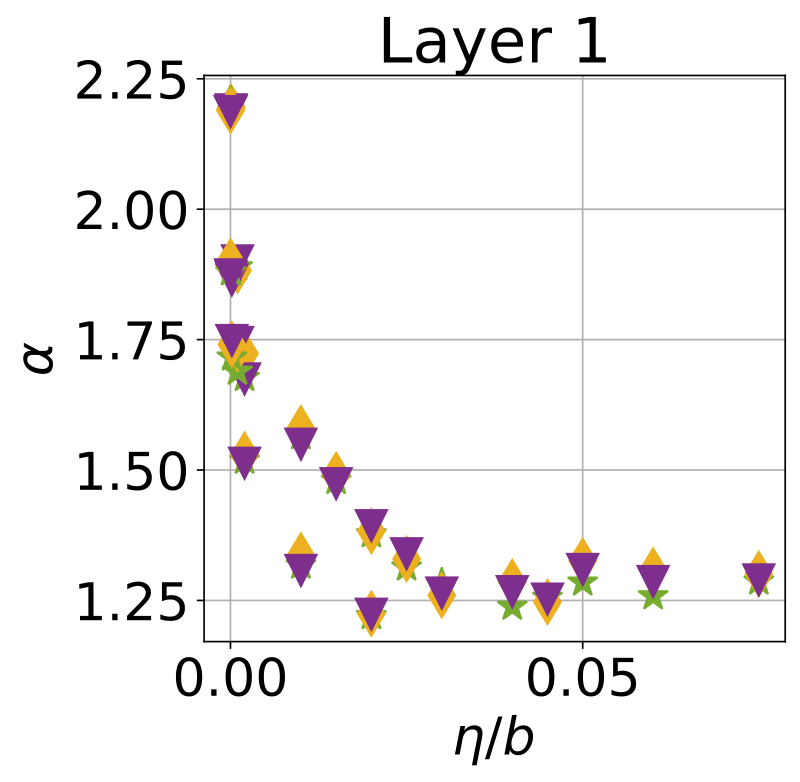}
    \includegraphics[width=0.3\columnwidth]{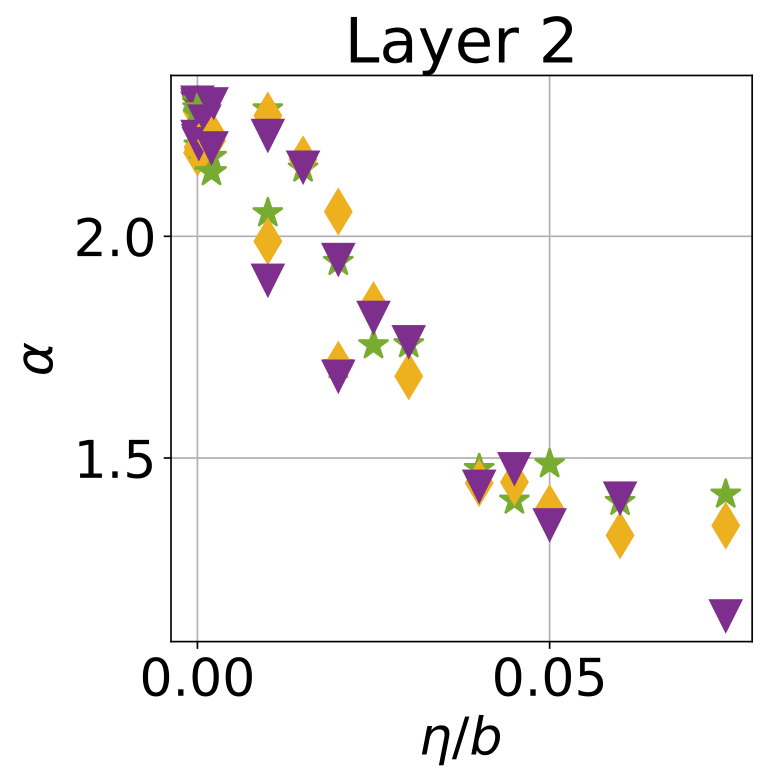}
    \includegraphics[width=0.3\columnwidth]{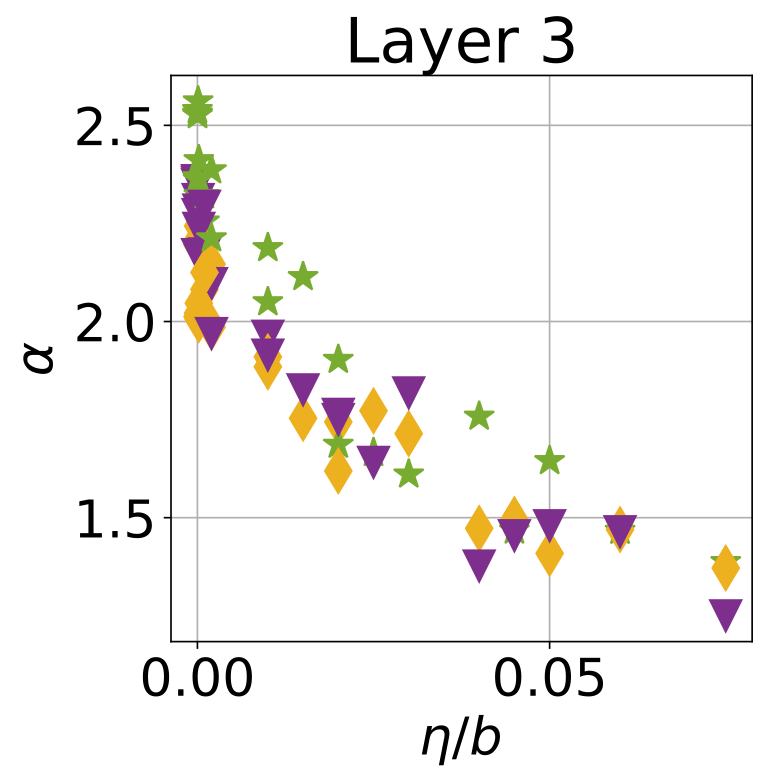}}
    \subfigure[CIFAR10]{
    \includegraphics[width=0.315\columnwidth]{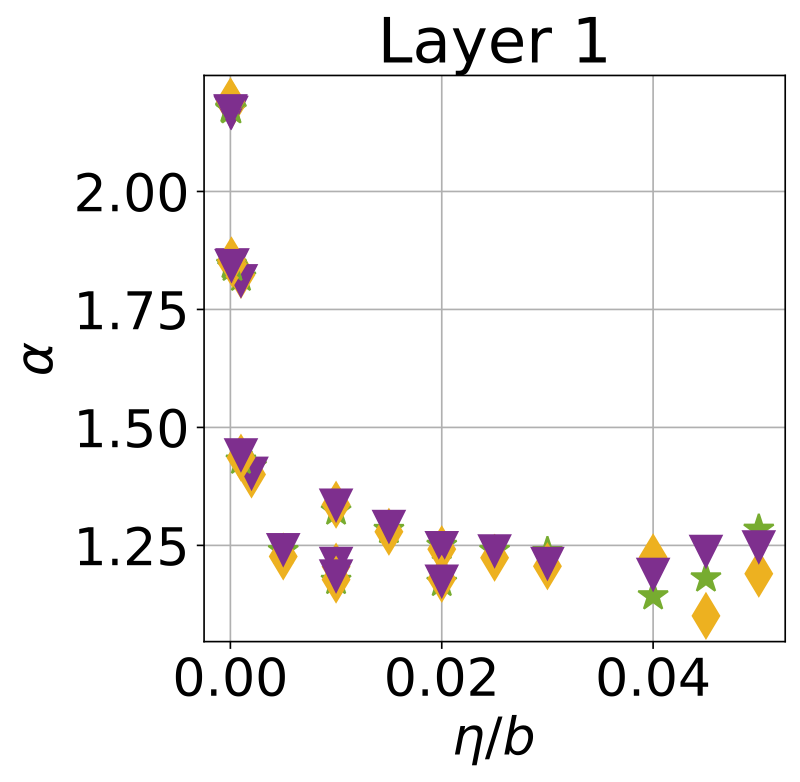}
    \includegraphics[width=0.3\columnwidth]{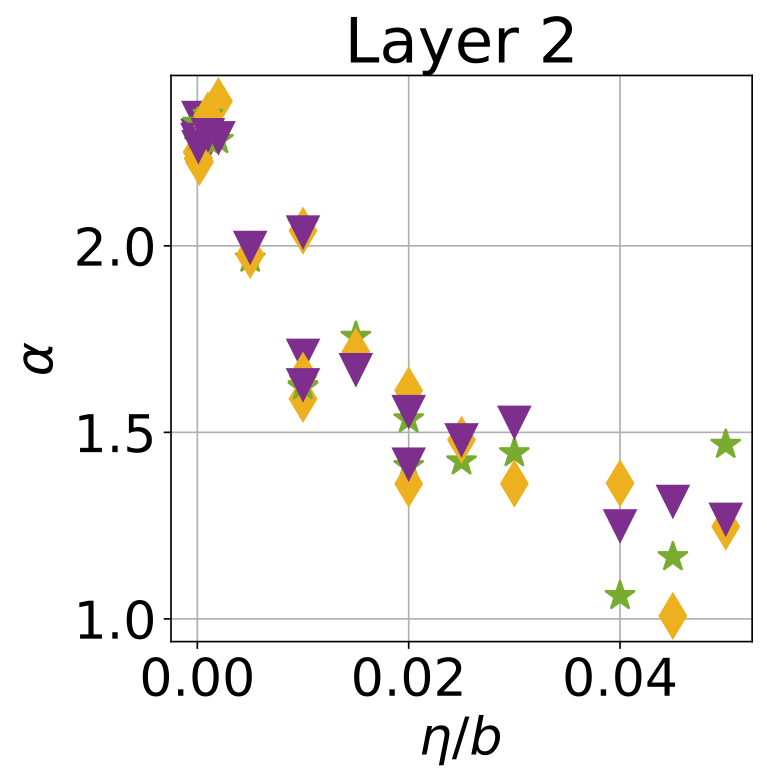}
    \includegraphics[width=0.3\columnwidth]{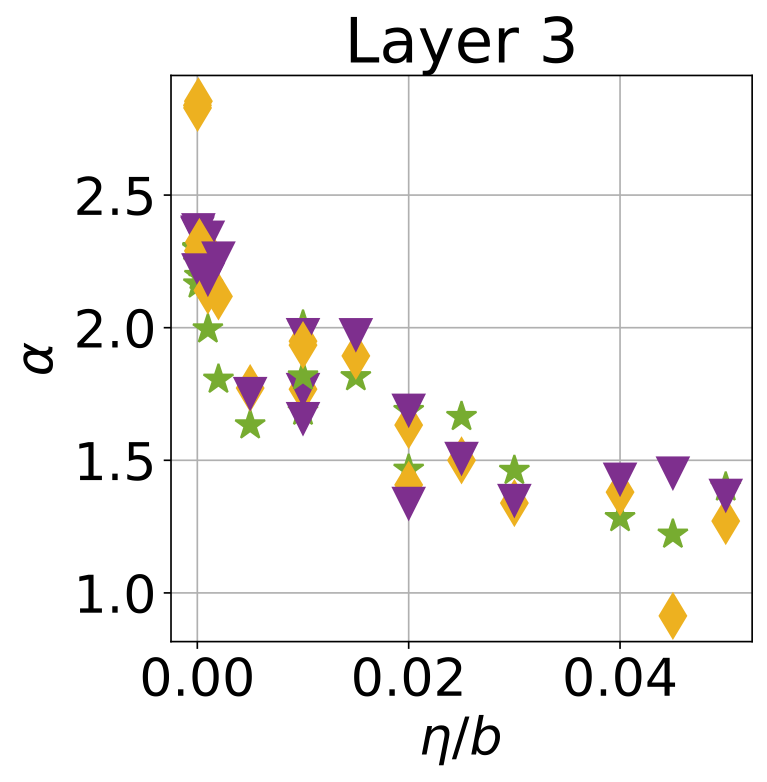}}
    \vspace{-10pt}
    \caption{Results on FCNs. Different markers represent different initializations with the same $\eta$, $b$.}
    \label{fig:exp_nn} 
\end{figure}

\begin{figure*}[t]
    \centering
    \includegraphics[width=0.25\columnwidth]{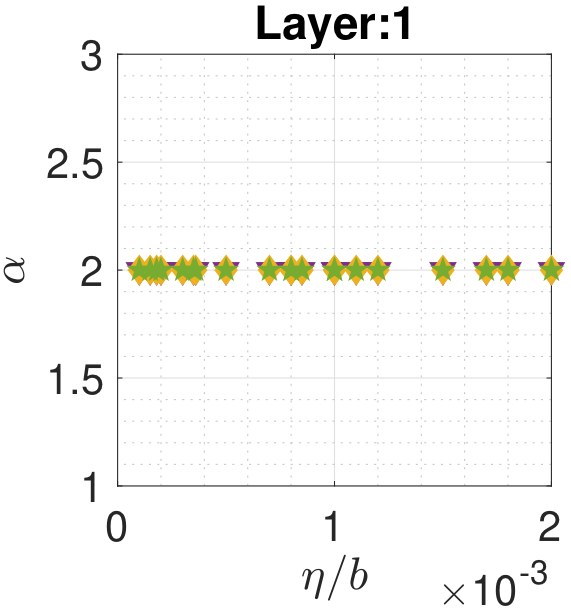}
    \includegraphics[width=0.25\columnwidth]{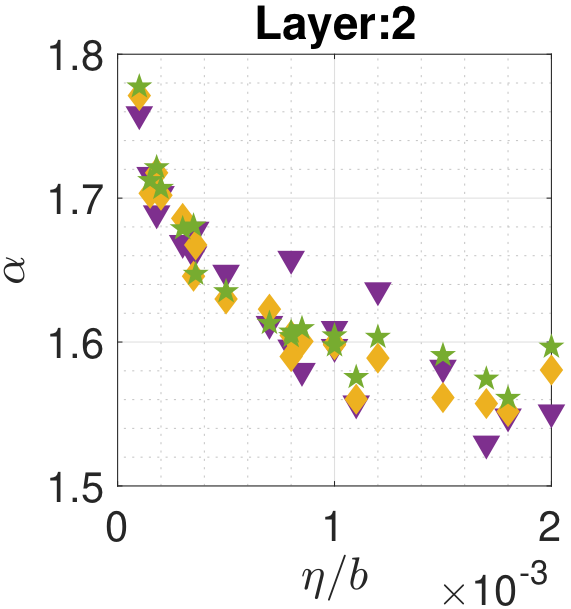}
    \includegraphics[width=0.25\columnwidth]{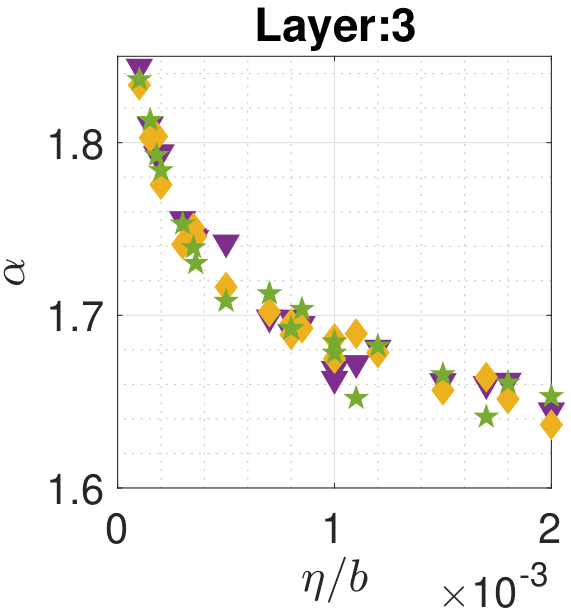}
    \includegraphics[width=0.25\columnwidth]{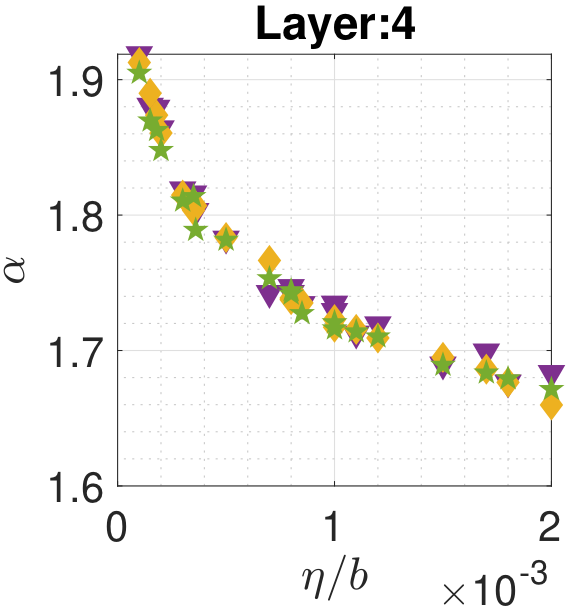}
    \includegraphics[width=0.25\columnwidth]{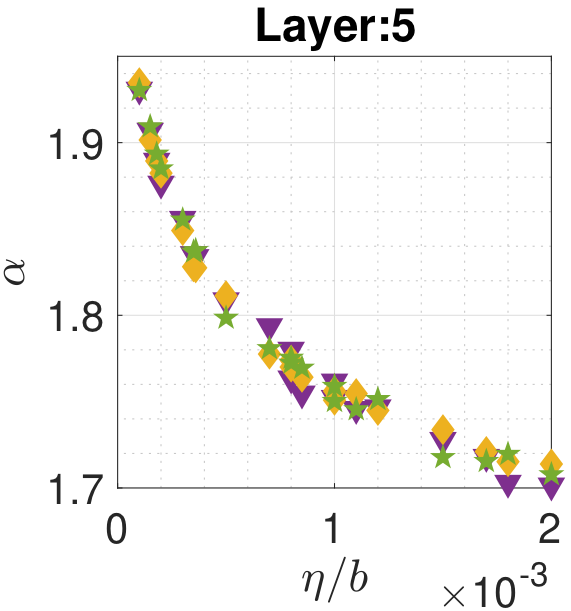}
    \includegraphics[width=0.25\columnwidth]{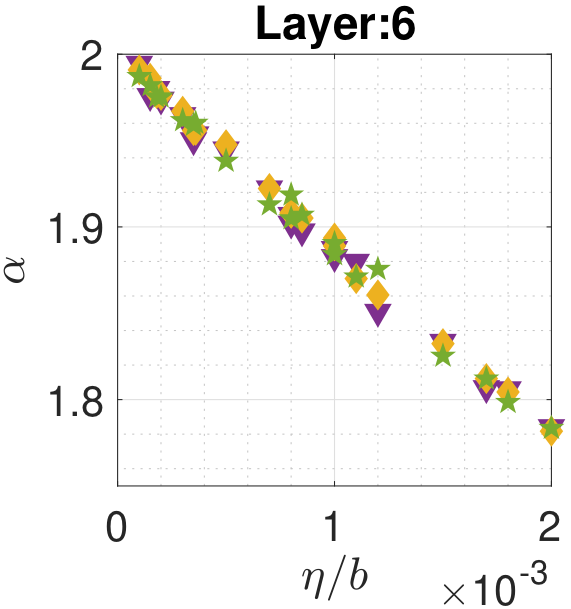}\\
    \includegraphics[width=0.25\columnwidth]{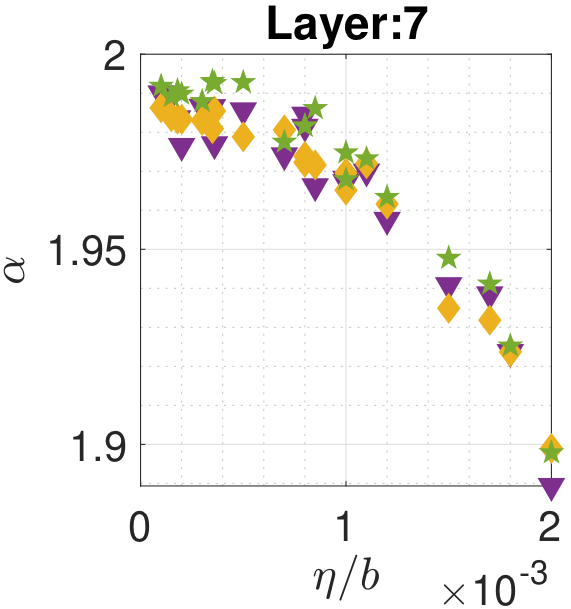}
    \includegraphics[width=0.25\columnwidth]{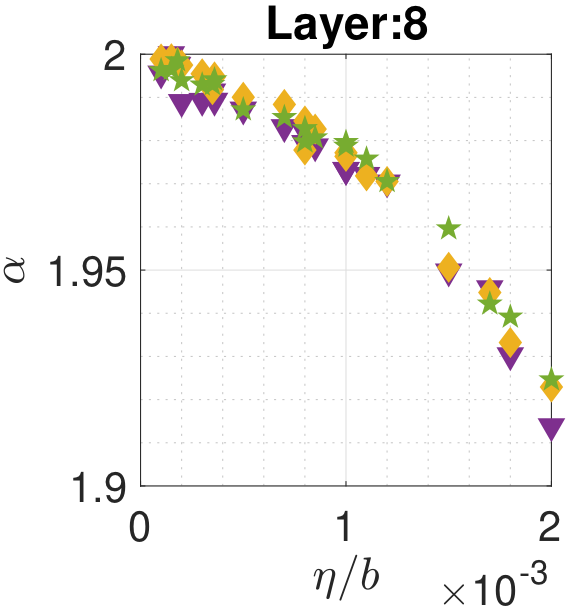}
    \includegraphics[width=0.25\columnwidth]{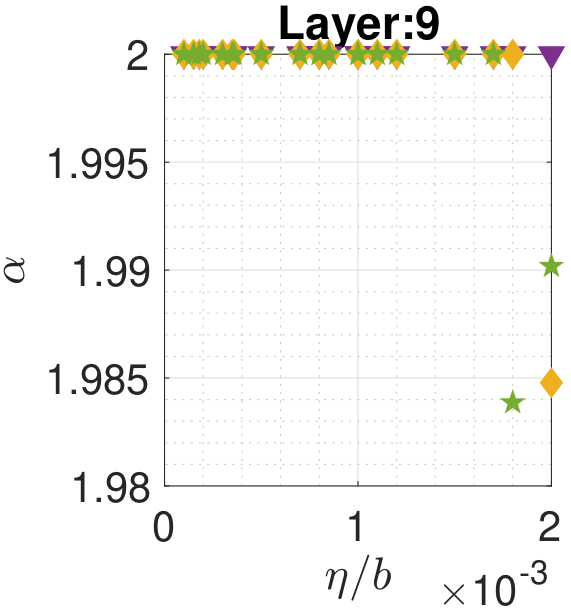}
    \includegraphics[width=0.25\columnwidth]{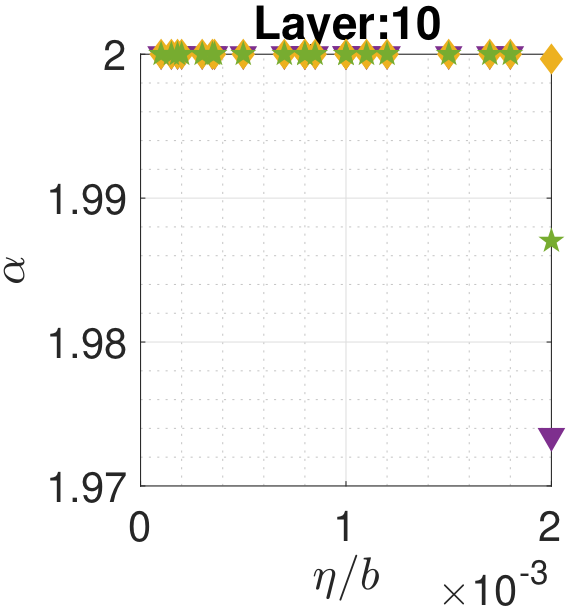}
    \includegraphics[width=0.25\columnwidth]{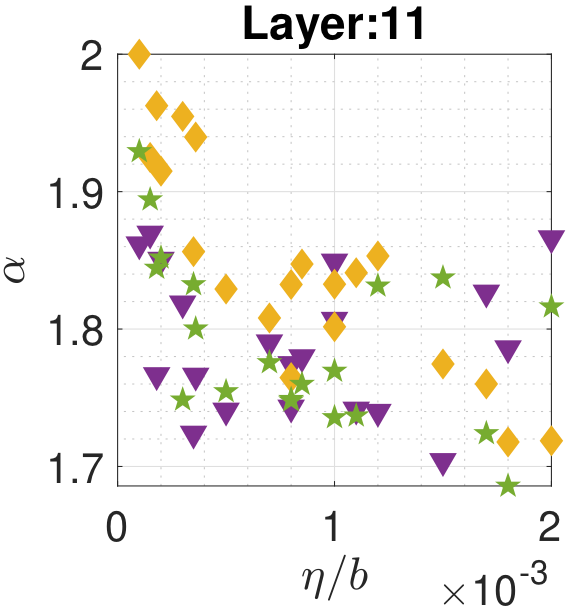}
    \includegraphics[width=0.25\columnwidth]{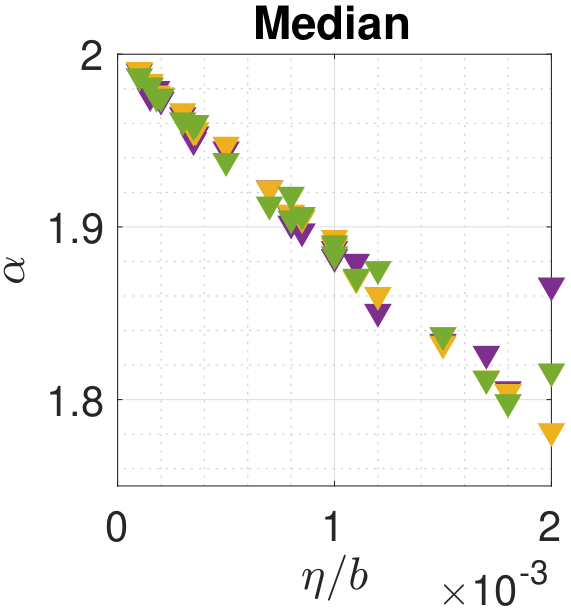}
    \caption{Results on VGG networks. The values of $\alpha$ that exceeded $2$ is truncated to $2$ for visualization purposes. Different markers represent different initializations.}
    \label{fig:vgg}
\end{figure*}

\section{Conclusion and Future Directions}
We studied the tail behavior of SGD
and showed that depending on $\eta$, $b$ and the curvature, the iterates can converge to a \emph{heavy-tailed} random variable in distribution. 
We further supported our theory with various experiments conducted on neural networks and illustrated that our results would also apply to more general settings and hence provide new insights about the behavior of SGD in deep learning.  
%
Our study also brings up a number of future directions. (i) Our proof techniques are for the streaming setting, where each sample is used only once. 
Extending our results to the finite-sum scenario and investigating the effects of finite-sample size on the tail-index would be an interesting future research direction. (ii) 
We suspect that the tail-index may have an impact on the time required to escape a saddle point and this can be investigated further {as another future research direction}. 
{(iii) Our work considers SGD with constant stepsize. Extending our analysis to adaptive methods and varying stepsizes is another interesting future research direction.}

\section*{Acknowledgements} 
M.G.'s research is supported in part by the grants Office of Naval Research Award Number N00014-21-1-2244, National Science Foundation (NSF) CCF-1814888, NSF DMS-2053485, NSF DMS-1723085. U.\c{S}.'s research is supported by the French government under management of Agence Nationale de la Recherche as part of the ``Investissements d'avenir program, reference ANR-19-P3IA-0001 (PRAIRIE 3IA Institute).
L.Z. is grateful to the support from a Simons Foundation Collaboration Grant and the grant NSF DMS-2053454 from the National Science Foundation.




\bibliography{heavy}

\begin{thebibliography}{61}
\providecommand{\natexlab}[1]{#1}
\providecommand{\url}[1]{\texttt{#1}}
\expandafter\ifx\csname urlstyle\endcsname\relax
  \providecommand{\doi}[1]{doi: #1}\else
  \providecommand{\doi}{doi: \begingroup \urlstyle{rm}\Url}\fi

\bibitem[Ali et~al.(2020)Ali, Dobriban, and Tibshirani]{ali2020implicit}
Alnur Ali, Edgar Dobriban, and Ryan~J Tibshirani.
\newblock The implicit regularization of stochastic gradient flow for least
  squares.
\newblock In \emph{Proceedings of the 37th International Conference on Machine
  Learning}, pages 233--244, 2020.

\bibitem[Alsmeyer(2016)]{alsmeyer2016stationary}
Gerold Alsmeyer.
\newblock On the stationary tail index of iterated random {L}ipschitz
  functions.
\newblock \emph{Stochastic Processes and their Applications}, 126\penalty0
  (1):\penalty0 209--233, 2016.

\bibitem[Alsmeyer and Mentemeier(2012)]{alsmeyer2012tail}
Gerold Alsmeyer and Sebastian Mentemeier.
\newblock Tail behaviour of stationary solutions of random difference
  equations: the case of regular matrices.
\newblock \emph{Journal of Difference Equations and Applications}, 18\penalty0
  (8):\penalty0 1305--1332, 2012.

\bibitem[Bauke(2007)]{bauke-power-law}
Heiko Bauke.
\newblock Parameter estimation for power-law distributions by maximum
  likelihood methods.
\newblock \emph{The European Physical Journal B}, 58\penalty0 (2):\penalty0
  167--173, 2007.

\bibitem[{Ben Arous} and Guionnet(2008)]{arous2008spectrum}
G{\'e}rard {Ben Arous} and Alice Guionnet.
\newblock The spectrum of heavy tailed random matrices.
\newblock \emph{Communications in Mathematical Physics}, 278\penalty0
  (3):\penalty0 715--751, 2008.

\bibitem[Bertoin(1996)]{bertoin1996}
Jean Bertoin.
\newblock \emph{L\'{e}vy Processes}.
\newblock Cambridge University Press, 1996.

\bibitem[Buraczewski et~al.(2012)Buraczewski, Damek, and
  Mirek]{buraczewski2012asymptotics}
Dariusz Buraczewski, Ewa Damek, and Mariusz Mirek.
\newblock Asymptotics of stationary solutions of multivariate stochastic
  recursions with heavy tailed inputs and related limit theorems.
\newblock \emph{Stochastic Processes and their Applications}, 122\penalty0
  (1):\penalty0 42--67, 2012.

\bibitem[Buraczewski et~al.(2014)Buraczewski, Damek, Guivarc'h, and
  Mentemeier]{buraczewski2014multidimensional}
Dariusz Buraczewski, Ewa Damek, Yves Guivarc'h, and Sebastian Mentemeier.
\newblock On multidimensional {M}andelbrot cascades.
\newblock \emph{Journal of Difference Equations and Applications}, 20\penalty0
  (11):\penalty0 1523--1567, 2014.

\bibitem[Buraczewski et~al.(2015)Buraczewski, Damek, and Przebinda]{bdp2015}
Dariusz Buraczewski, Ewa Damek, and Tomasz Przebinda.
\newblock On the rate of convergence in the {K}esten renewal theorem.
\newblock \emph{Electronic Journal of Probaiblity}, 20\penalty0 (22):\penalty0
  1--35, 2015.

\bibitem[Buraczewski et~al.(2016)Buraczewski, Damek, and
  Mikosch]{buraczewski2016stochastic}
Dariusz Buraczewski, Ewa Damek, and Thomas Mikosch.
\newblock \emph{Stochastic Models with Power-Law Tails}.
\newblock Springer, 2016.

\bibitem[Chaudhari and Soatto(2018)]{chaudhari2018stochastic}
Pratik Chaudhari and Stefano Soatto.
\newblock Stochastic gradient descent performs variational inference, converges
  to limit cycles for deep networks.
\newblock In \emph{International Conference on Learning Representations}, 2018.

\bibitem[Cheng et~al.(2020)Cheng, Yin, Bartlett, and
  Jordan]{cheng2019stochastic}
Xiang Cheng, Dong Yin, Peter~L Bartlett, and Michael~I Jordan.
\newblock Stochastic gradient and {L}angevin processes.
\newblock In \emph{Proceedings of the 37th International Conference on Machine
  Learning}, pages 1810--1819, 2020.

\bibitem[Clauset et~al.(2009)Clauset, Shalizi, and Newman]{clauset2009power}
Aaron Clauset, Cosma~Rohilla Shalizi, and Mark~EJ Newman.
\newblock Power-law distributions in empirical data.
\newblock \emph{SIAM Review}, 51\penalty0 (4):\penalty0 661--703, 2009.

\bibitem[De~Bortoli et~al.(2020)De~Bortoli, Durmus, Fontaine, and
  \c{S}im\c{s}ekli]{de2020quantitative}
Valentin De~Bortoli, Alain Durmus, Xavier Fontaine, and Umut \c{S}im\c{s}ekli.
\newblock Quantitative propagation of chaos for {SGD} in wide neural networks.
\newblock In \emph{Advances in Neural Information Processing Systems},
  volume~33, 2020.

\bibitem[Defazio and Bottou(2019)]{defazio_2019}
Aaron Defazio and Leon Bottou.
\newblock On the ineffectiveness of variance reduced optimization for deep
  learning.
\newblock In \emph{Advances in Neural Information Processing Systems}, pages
  1755--1765, 2019.

\bibitem[Diaconis and Freedman(1999)]{diaconis1999iterated}
Persi Diaconis and David Freedman.
\newblock Iterated random functions.
\newblock \emph{SIAM Review}, 41\penalty0 (1):\penalty0 45--76, 1999.

\bibitem[Dieuleveut et~al.(2017)Dieuleveut, Flammarion, and
  Bach]{dieuleveut2017harder}
Aymeric Dieuleveut, Nicolas Flammarion, and Francis Bach.
\newblock Harder, better, faster, stronger convergence rates for least-squares
  regression.
\newblock \emph{The Journal of Machine Learning Research}, 18\penalty0
  (1):\penalty0 3520--3570, 2017.

\bibitem[Dieuleveut et~al.(2020)Dieuleveut, Durmus, and
  Bach]{dieuleveut2017bridging}
Aymeric Dieuleveut, Alain Durmus, and Francis Bach.
\newblock Bridging the gap between constant step size stochastic gradient
  descent and {M}arkov chains.
\newblock \emph{Annals of Statistics}, 48\penalty0 (3):\penalty0 1348--1382,
  2020.

\bibitem[Dinh et~al.(2017)Dinh, Pascanu, Bengio, and Bengio]{dinh2017sharp}
Laurent Dinh, Razvan Pascanu, Samy Bengio, and Yoshua Bengio.
\newblock Sharp minima can generalize for deep nets.
\newblock In \emph{Proceedings of the 34th International Conference on Machine
  Learning-Volume 70}, pages 1019--1028. JMLR. org, 2017.

\bibitem[Fink and Kl{\"u}ppelberg(2011)]{fink2011fractional}
Holger Fink and Claudia Kl{\"u}ppelberg.
\newblock Fractional {L}{\'e}vy-driven {O}rnstein--{U}hlenbeck processes and
  stochastic differential equations.
\newblock \emph{Bernoulli}, 17\penalty0 (1):\penalty0 484--506, 2011.

\bibitem[Frostig et~al.(2015)Frostig, Ge, Kakade, and
  Sidford]{frostig2015competing}
Roy Frostig, Rong Ge, Sham~M Kakade, and Aaron Sidford.
\newblock Competing with the empirical risk minimizer in a single pass.
\newblock In \emph{Conference on Learning Theory}, pages 728--763, 2015.

\bibitem[Gao et~al.(2021)Gao, G{\"u}rb{\"u}zbalaban, and Zhu]{gao2018global}
Xuefeng Gao, Mert G{\"u}rb{\"u}zbalaban, and Lingjiong Zhu.
\newblock Global convergence of stochastic gradient hamiltonian monte carlo for
  non-convex stochastic optimization: Non-asymptotic performance bounds and
  momentum-based acceleration.
\newblock \emph{To Appear, Operations Research}, 2021.

\bibitem[Goldie(1991)]{goldie1991implicit}
Charles~M Goldie.
\newblock Implicit renewal theory and tails of solutions of random equations.
\newblock \emph{Annals of Applied Probability}, 1\penalty0 (1):\penalty0
  126--166, 1991.

\bibitem[Goldstein et~al.(2004)Goldstein, Morris, and
  Yen]{goldstein2004problems}
Michel~L Goldstein, Steven~A Morris, and Gary~G Yen.
\newblock Problems with fitting to the power-law distribution.
\newblock \emph{The European Physical Journal B-Condensed Matter and Complex
  Systems}, 41\penalty0 (2):\penalty0 255--258, 2004.

\bibitem[Henriksen and Ward(2020)]{henriksen2020concentration}
Amelia Henriksen and Rachel Ward.
\newblock Concentration inequalities for random matrix products.
\newblock \emph{Linear Algebra and its Applications}, 594:\penalty0 81--94,
  2020.

\bibitem[Hinton et~al.(2012)Hinton, Srivastava, and Swersky]{hinton2012neural}
Geoffrey Hinton, Nitish Srivastava, and Kevin Swersky.
\newblock Overview of mini-batch gradient descent.
\newblock \emph{Neural Networks for Machine Learning}, Lecture 6a, 2012.
\newblock URL
  \url{http://www.cs.toronto.edu/~hinton/coursera/lecture6/lec6.pdf}.

\bibitem[Hochreiter and Schmidhuber(1997)]{hochreiter1997flat}
Sepp Hochreiter and J{\"u}rgen Schmidhuber.
\newblock Flat minima.
\newblock \emph{Neural Computation}, 9\penalty0 (1):\penalty0 1--42, 1997.

\bibitem[Hodgkinson and Mahoney(2020)]{hodgkinson2020multiplicative}
Liam Hodgkinson and Michael~W Mahoney.
\newblock Multiplicative noise and heavy tails in stochastic optimization.
\newblock \emph{arXiv preprint arXiv:2006.06293}, June 2020.

\bibitem[Hu et~al.(2019)Hu, Li, Li, and Liu]{hu2017diffusion}
Wenqing Hu, Chris~Junchi Li, Lei Li, and Jian-Guo Liu.
\newblock On the diffusion approximation of nonconvex stochastic gradient
  descent.
\newblock \emph{Annals of Mathematical Science and Applications}, 4\penalty0
  (1):\penalty0 3--32, 2019.

\bibitem[Huang et~al.(2020)Huang, Niles-Weed, Tropp, and Ward]{huang2020matrix}
De~Huang, Jonathan Niles-Weed, Joel~A. Tropp, and Rachel Ward.
\newblock Matrix concentration for products.
\newblock \emph{arXiv preprint arXiv:2003.05437}, 2020.

\bibitem[Jain et~al.(2017)Jain, Kakade, Kidambi, Netrapalli, and
  Sidford]{jain2017accelerating}
Prateek Jain, Sham~M Kakade, Rahul Kidambi, Praneeth Netrapalli, and Aaron
  Sidford.
\newblock Accelerating stochastic gradient descent.
\newblock In \emph{Proc. STAT}, volume 1050, page~26, 2017.

\bibitem[Jastrz{\k{e}}bski et~al.(2017)Jastrz{\k{e}}bski, Kenton, Arpit,
  Ballas, Fischer, Bengio, and Storkey]{jastrzkebski2017three}
Stanis{\l}aw Jastrz{\k{e}}bski, Zachary Kenton, Devansh Arpit, Nicolas Ballas,
  Asja Fischer, Yoshua Bengio, and Amos Storkey.
\newblock Three factors influencing minima in {SGD}.
\newblock \emph{arXiv preprint arXiv:1711.04623}, 2017.

\bibitem[Johnson and Zhang(2013)]{johnson2013accelerating}
Rie Johnson and Tong Zhang.
\newblock Accelerating stochastic gradient descent using predictive variance
  reduction.
\newblock In \emph{Advances in Neural Information Processing Systems}, pages
  315--323, 2013.

\bibitem[Keskar and Socher(2017)]{keskar2017improving}
Nitish~Shirish Keskar and Richard Socher.
\newblock Improving generalization performance by switching from {A}dam to
  {SGD}.
\newblock \emph{arXiv preprint arXiv:1712.07628}, 2017.

\bibitem[Keskar et~al.(2017)Keskar, Mudigere, Nocedal, Smelyanskiy, and
  Tang]{keskar2016large}
Nitish~Shirish Keskar, Dheevatsa Mudigere, Jorge Nocedal, Mikhail Smelyanskiy,
  and Ping Tak~Peter Tang.
\newblock On large-batch training for deep learning: Generalization gap and
  sharp minima.
\newblock In \emph{5th International Conference on Learning Representations,
  ICLR 2017}, 2017.

\bibitem[Kesten(1973)]{kesten1973random}
Harry Kesten.
\newblock Random difference equations and renewal theory for products of random
  matrices.
\newblock \emph{Acta Mathematica}, 131:\penalty0 207--248, 1973.

\bibitem[L{\'e}vy(1937)]{paul1937theorie}
Paul L{\'e}vy.
\newblock Th{\'e}orie de l'addition des variables al{\'e}atoires.
\newblock \emph{Gauthiers-Villars, Paris}, 1937.

\bibitem[Lewkowycz et~al.(2020)Lewkowycz, Bahri, Dyer, Sohl-Dickstein, and
  Gur-Ari]{lewkowycz2020large}
Aitor Lewkowycz, Yasaman Bahri, Ethan Dyer, Jascha Sohl-Dickstein, and Guy
  Gur-Ari.
\newblock The large learning rate phase of deep learning: the catapult
  mechanism.
\newblock \emph{arXiv preprint arXiv:2003.02218}, 2020.

\bibitem[Li et~al.(2017)Li, Tai, and E]{pmlr-v70-li17f}
Qianxiao Li, Cheng Tai, and Weinan E.
\newblock Stochastic modified equations and adaptive stochastic gradient
  algorithms.
\newblock In \emph{Proceedings of the 34th International Conference on Machine
  Learning}, pages 2101--2110, 06--11 Aug 2017.

\bibitem[Mandt et~al.(2016)Mandt, Hoffman, and Blei]{mandt2016variational}
Stephan Mandt, Matthew~D. Hoffman, and David~M. Blei.
\newblock A variational analysis of stochastic gradient algorithms.
\newblock In \emph{International Conference on Machine Learning}, pages
  354--363, 2016.

\bibitem[Martin and Mahoney(2019)]{martin2019traditional}
Charles~H Martin and Michael~W Mahoney.
\newblock Traditional and heavy-tailed self regularization in neural network
  models.
\newblock In \emph{Proceedings of the 36th International Conference on Machine
  Learning}, 2019.

\bibitem[Mirek(2011)]{mirek2011heavy}
Mariusz Mirek.
\newblock Heavy tail phenomenon and convergence to stable laws for iterated
  {L}ipschitz maps.
\newblock \emph{Probability Theory and Related Fields}, 151\penalty0
  (3-4):\penalty0 705--734, 2011.

\bibitem[Mohammadi et~al.(2015)Mohammadi, Mohammadpour, and
  Ogata]{mohammadi2015estimating}
Mohammad Mohammadi, Adel Mohammadpour, and Hiroaki Ogata.
\newblock On estimating the tail index and the spectral measure of multivariate
  $\alpha$-stable distributions.
\newblock \emph{Metrika}, 78\penalty0 (5):\penalty0 549--561, 2015.

\bibitem[Newman(1986)]{newman1986distribution}
Charles~M Newman.
\newblock The distribution of {L}yapunov exponents: Exact results for random
  matrices.
\newblock \emph{Communications in Mathematical Physics}, 103\penalty0
  (1):\penalty0 121--126, 1986.

\bibitem[Nguyen et~al.(2019)Nguyen, \c{S}im\c{s}ekli, G\"{u}rb\"{u}zbalaban,
  and Richard]{nguyen2019first}
Thanh~Huy Nguyen, Umut \c{S}im\c{s}ekli, Mert G\"{u}rb\"{u}zbalaban, and
  Ga{\"e}l Richard.
\newblock First exit time analysis of stochastic gradient descent under
  heavy-tailed gradient noise.
\newblock In \emph{Advances in Neural Information Processing Systems}, pages
  273--283, 2019.

\bibitem[{\O}ksendal(2013)]{oksendal2013stochastic}
Bernt {\O}ksendal.
\newblock \emph{Stochastic Differential Equations: An Introduction with
  Applications}.
\newblock Springer Science \& Business Media, 2013.

\bibitem[Panigrahi et~al.(2019)Panigrahi, Somani, Goyal, and
  Netrapalli]{panigrahi2019non}
Abhishek Panigrahi, Raghav Somani, Navin Goyal, and Praneeth Netrapalli.
\newblock Non-{G}aussianity of stochastic gradient noise.
\newblock \emph{arXiv preprint arXiv:1910.09626}, 2019.

\bibitem[Paulauskas and Vai{\v{c}}iulis(2011)]{paulauskas2011once}
Vygantas Paulauskas and Marijus Vai{\v{c}}iulis.
\newblock Once more on comparison of tail index estimators.
\newblock \emph{arXiv preprint arXiv:1104.1242}, 2011.

\bibitem[Pavlyukevich(2007)]{Pavlyukevich_2007}
Ilya Pavlyukevich.
\newblock Cooling down {L}\'{e}vy flights.
\newblock \emph{Journal of Physics A: Mathematical and Theoretical},
  40\penalty0 (41):\penalty0 12299–12313, 2007.

\bibitem[Shalev-Shwartz and Ben-David(2014)]{shalev2014understanding}
Shai Shalev-Shwartz and Shai Ben-David.
\newblock \emph{Understanding Machine Learning: From Theory to Algorithms}.
\newblock Cambridge University Press, 2014.

\bibitem[Simonyan and Zisserman(2015)]{simonyan2014very}
Karen Simonyan and Andrew Zisserman.
\newblock Very deep convolutional networks for large-scale image recognition.
\newblock In \emph{3rd International Conference on Learning Representations,
  ICLR 2015}, 2015.

\bibitem[{\c{S}}im{\c{s}}ekli et~al.(2019{\natexlab{a}}){\c{S}}im{\c{s}}ekli,
  G{\"u}rb{\"u}zbalaban, Nguyen, Richard, and Sagun]{csimcsekli2019heavy}
Umut {\c{S}}im{\c{s}}ekli, Mert G{\"u}rb{\"u}zbalaban, Thanh~Huy Nguyen,
  Ga{\"e}l Richard, and Levent Sagun.
\newblock On the heavy-tailed theory of stochastic gradient descent for deep
  neural networks.
\newblock \emph{arXiv preprint arXiv:1912.00018}, 2019{\natexlab{a}}.

\bibitem[{\c{S}}im{\c{s}}ekli et~al.(2019{\natexlab{b}}){\c{S}}im{\c{s}}ekli,
  Sagun, and G{\"u}rb{\"u}zbalaban]{pmlr-v97-simsekli19a}
Umut {\c{S}}im{\c{s}}ekli, Levent Sagun, and Mert G{\"u}rb{\"u}zbalaban.
\newblock A tail-index analysis of stochastic gradient noise in deep neural
  networks.
\newblock In \emph{International Conference on Machine Learning}, pages
  5827--5837, 2019{\natexlab{b}}.

\bibitem[{\c{S}}im{\c{s}}ekli et~al.(2020){\c{S}}im{\c{s}}ekli, Sener,
  Deligiannidis, and Erdogdu]{csimcsekli2020hausdorff}
Umut {\c{S}}im{\c{s}}ekli, Ozan Sener, George Deligiannidis, and Murat~A
  Erdogdu.
\newblock Hausdorff dimension, stochastic differential equations, and
  generalization in neural networks.
\newblock In \emph{Advances in Neural Information Processing Systems},
  volume~33, 2020.

\bibitem[Srikant and Ying(2019)]{srikant2019finite}
Rayadurgam Srikant and Lei Ying.
\newblock Finite-time error bounds for linear stochastic approximation and {TD}
  learning.
\newblock In \emph{Conference on Learning Theory}, pages 2803--2830. PMLR,
  2019.

\bibitem[Tzagkarakis et~al.(2018)Tzagkarakis, Nolan, and
  Tsakalides]{tzagkarakis2018compressive}
George Tzagkarakis, John~P Nolan, and Panagiotis Tsakalides.
\newblock Compressive sensing of temporally correlated sources using isotropic
  multivariate stable laws.
\newblock In \emph{2018 26th European Signal Processing Conference (EUSIPCO)},
  pages 1710--1714. IEEE, 2018.

\bibitem[Villani(2009)]{villani2008optimal}
C{\'e}dric Villani.
\newblock \emph{Optimal Transport: Old and New}.
\newblock Springer, Berlin, 2009.

\bibitem[Xie et~al.(2021)Xie, Sato, and Sugiyama]{xie2021a}
Zeke Xie, Issei Sato, and Masashi Sugiyama.
\newblock A diffusion theory for deep learning dynamics: Stochastic gradient
  descent exponentially favors flat minima.
\newblock In \emph{International Conference on Learning Representations}, 2021.

\bibitem[Zhang et~al.(2020)Zhang, Karimireddy, Veit, Kim, Reddi, Kumar, and
  Sra]{zhang2019adam}
Jingzhao Zhang, Sai~Praneeth Karimireddy, Andreas Veit, Seungyeon Kim, Sashank
  Reddi, Sanjiv Kumar, and Suvrit Sra.
\newblock Why are adaptive methods good for attention models?
\newblock In \emph{Advances in Neural Information Processing Systems
  (NeurIPS)}, volume~33, 2020.

\bibitem[Zhou et~al.(2020)Zhou, Feng, Ma, Xiong, Hoi, and E]{zhou2020towards}
Pan Zhou, Jiashi Feng, Chao Ma, Caiming Xiong, Steven Hoi, and Weinan E.
\newblock Towards theoretically understanding why {SGD} generalizes better than
  {ADAM} in deep learning.
\newblock In \emph{Advances in Neural Information Processing Systems
  (NeurIPS)}, volume~33, 2020.

\bibitem[Zhu et~al.(2019)Zhu, Wu, Yu, Wu, and Ma]{zhu2018anisotropic}
Zhanxing Zhu, Jingfeng Wu, Bing Yu, Lei Wu, and Jinwen Ma.
\newblock The anisotropic noise in stochastic gradient descent: Its behavior of
  escaping from minima and regularization effects.
\newblock In \emph{Proceedings of the 36th International Conference on Machine
  Learning}, 2019.

\end{thebibliography}

\appendix

\section*{}

%
%


%

\newpage
\section{A Note on Stochastic Differential Equation Representations for SGD}\label{sec-appendix-sde}


In recent years, a popular approach for analyzing the behavior of SGD has been viewing it as a discretization of a continuous-time stochastic process that can be represented via a stochastic differential equation (SDE) \citep{mandt2016variational,jastrzkebski2017three,pmlr-v70-li17f,hu2017diffusion,zhu2018anisotropic,chaudhari2018stochastic,pmlr-v97-simsekli19a}. While these SDEs have been useful for understanding different properties of SGD, their differences and functionalities have not been clearly understood. In this section, in light of our theoretical results, we will discuss in which situation their choice would be more appropriate. We will restrict ourselves to the case where $f(x)$ is a quadratic function; however, the discussion can be extended to more general $f$.

The SDE approximations are often motivated by first rewriting the SGD recursion as follows:
\begin{align}
x_{k+1} = x_k - \eta \nabla \tilde{f}_{k+1}\left(x_k\right) =  x_k - \eta \nabla f\left(x_k\right) + \eta U_{k+1}(x_k), \label{eqn:sgd_noise}
\end{align} 
where $U_k(x) := \nabla \tilde{f}_k (x)- \nabla f(x)$ is called the `stochastic gradient noise'. Then, based on certain statistical assumptions on $U_k$, we can view (\ref{eqn:sgd_noise}) as a discretization of an SDE. 
For instance, if we assume that the gradient noise follows a Gaussian distribution, whose covariance does not depend on the iterate $x_k$, i.e., $\eta U_k \approx \sqrt{\eta} Z_k$ where $Z_k \sim \mathcal{N}(0, \sigma_z \eta I)$ for some constant $\sigma_z >0$, we can see (\ref{eqn:sgd_noise}) as the Euler-Maruyama discretization of the following SDE with stepsize $\eta$ \citep{mandt2016variational}:
\begin{align}
\rmd x_t = - \nabla f(x_t)\rmd t + \sqrt{\eta \sigma_z} \rmd \Bm_t,
\end{align}
where $\Bm_t$ denotes the $d$-dimensional standard Brownian motion. 
This process is called the Ornstein-Uhlenbeck (OU) process (see e.g. \citet{oksendal2013stochastic}), whose invariant measure is a Gaussian distribution. 
We argue that this process can be a good proxy to (\ref{eq-stoc-grad-2}) only when $\alpha \geq 2$, since otherwise the SGD iterates will exhibit heavy-tails, whose behavior cannot be captured by a Gaussian distribution. As we illustrated in Section~\ref{sec:exps}, to obtain large $\alpha$, the stepsize $\eta$ needs to be small and/or the batch-size $b$ needs to be large.
However, it is clear that this approximation will fall short when the system exhibits heavy tails, i.e., $\alpha<2$. Therefore, for the large $\eta/b$ regime, which appears to be more interesting since it often yields improved test performance \citep{jastrzkebski2017three}, this approximation would be inaccurate for understanding the behavior of SGD. 
This problem mainly stems from the fact that the additive isotropic noise assumption results in a deterministic $M_k$ matrix for all $k$. Since there is no \emph{multiplicative noise} term, this representation cannot capture a potential heavy-tailed behavior. 
%

%


A natural extension of the state-independent Gaussian noise assumption is to incorporate the covariance structure of $U_k$. In our linear regression problem, we can easily see that the covariance matrix of the gradient noise has the following form: 
\begin{align}
\Sigma_U(x) = \mathrm{Cov}(U_k|x) = \frac{\sigma^2}{b} \mathrm{diag}(x \circ x),
\end{align}
where $\circ$ denotes element-wise multiplication and $\sigma^2$ is the variance of the data points. Therefore, we can extend the previous assumption by assuming $Z_k | x \sim \mathcal{N}(0, \eta \Sigma_U(x))$. It has been observed that this approximation yields a more accurate representation \citep{cheng2019stochastic,ali2020implicit,jastrzkebski2017three}. Using this assumption in (\ref{eqn:sgd_noise}), the SGD recursion coincides with the Euler-Maruyama discretization of the following SDE: 
\begin{align}
d x_t &= -\nabla f(x_t)dt + \sqrt{\eta \Sigma_U(x_t)} d\Bm_t \nonumber\\
&\eqdist -\left(A^\top A x_t - A^\top y\right)dt + \sqrt{\frac{\sigma^2\eta}{b}} \mathrm{diag}(x_t) d\Bm_t,
\end{align}
where $\eqdist$ denotes equality in distribution. The stochasticity in such SDEs is called often called \emph{multiplicative}. Let us illustrate this property by discretizing this process and by using the definition of the gradient and the covariance matrix, we observe that (noting that $N_k \sim \mathcal{N}(0,I)$)
\begin{align}
x_{k+1} &= x_{k} - \eta\left( A^\top A x_{k} - A^\top y\right) +  \sqrt{\frac{\sigma^2\eta^2}{b}} \mathrm{diag}(x_k) N_{k+1}\nonumber \\
&= \left(I - \eta A^\top A +  \sqrt{\sigma^2\eta^2/b} \> \mathrm{diag}( N_{k+1}) \right) x_{k} - \eta A^\top y,
\end{align}
where we can clearly see the multiplicative effect of the noise, as indicated by its name. On the other hand, we can observe that, thanks to the multiplicative structure, this process would be able to capture the potential heavy-tailed structure of SGD. However, there are two caveats. The first one is that, in the case of linear regression, the process is called a geometric (or modified) Ornstein-Uhlenbeck process which is an extension of geometric Brownian motion. One can show that the distribution of the process at any time $t$ will have lognormal tails. 
Hence it will be accurate only when the tail-index $\alpha$ is close to the one of the lognormal distribution. 
The second caveat is that, for a more general cost function $f$, the covariance matrix is more complicated and hence the invariant measure of the process cannot be found analytically, hence analyzing these processes for a general $f$ can be as challenging as directly analyzing the behavior of SGD.  

The third way of modeling the gradient noise is based on assuming that it is heavy-tailed. In particular, we can assume that $\eta U_k \approx \eta^{1/\alpha} L_k $ where $[L_k]_i \sim \sas(\sigma_L \eta^{(\alpha-1)/\alpha}) $ for all $i=1,\dots,d$. Under this assumption the SGD recursion coincides with the Euler discretization of the following L\'{e}vy-driven SDE \citep{pmlr-v97-simsekli19a}:
\begin{align}
d x_t = -\nabla f(x_t)dt + \sigma_L \eta^{(\alpha-1)/\alpha} d\Lm_t,
\end{align}
where $\Lm_t$ denotes the $\alpha$-stable L\'{e}vy process with independent components (see Section~\ref{Tech:Levy}
for technical background on L\'{e}vy processes and
in particular $\alpha$-stable L\'{e}vy processes). In the case of linear regression, this processes is called a fractional OU process \citep{fink2011fractional}, whose invariant measure is also an $\alpha$-stable distribution with the same tail-index $\alpha$. Hence, even though it is based on an isotropic, state-independent noise assumption, in the case of large $\eta/b$ regime, this approach can mimic the heavy-tailed behavior of the system with the exact tail-index $\alpha$. 
On the other hand, \citet{buraczewski2016stochastic} (Theorem 1.7 and 1.16) showed that if $U_k$ is assumed to heavy tailed with index $\alpha$ (not necessarily $\sas$) then the process $x_k$ will inherit the same tails and the ergodic averages will still converge to an $\sas$ random variable {in distribution}, hence generalizing the conclusions of the $\sas$ assumption to the case where $U_k$ follows an arbitrary heavy-tailed distribution.

\subsection{Technical background: L\'{e}vy processes}\label{Tech:Levy}
L\'{e}vy motions (processes) are stochastic processes with independent and stationary increments,
which include Brownian motions as a special case,
and in general may have heavy-tailed distributions
(see e.g. \citet{bertoin1996}
for a survey).
Symmetric $\alpha$-stable L\'{e}vy motion
is a L\'{e}vy motion
whose time increments are symmetric $\alpha$-stable distributed.
We define $\Lm_{t}$, a $d$-dimensional 
symmetric $\alpha$-stable L\'{e}vy motion as follows. 
Each component of $\Lm_{t}$ is an independent scalar $\alpha$-stable L\'{e}vy process defined as follows:

(i) $\Lm_{0}=0$ almost surely;

(ii) For any $t_{0}<t_{1}<\cdots<t_{N}$, the increments $\Lm_{t_{n}}-\Lm_{t_{n-1}}$
are independent, $n=1,2,\ldots,N$;

(iii) The difference $\Lm_{t}-\Lm_{s}$ and $\Lm_{t-s}$
have the same distribution: $\mathcal{S}\alpha\mathcal{S}((t-s)^{1/\alpha})$ for $s<t$;

(iv) $\Lm_{t}$ has stochastically continuous sample paths, i.e.
for any $\delta>0$ and $s\geq 0$, $\mathbb{P}(|\Lm_{t}-\Lm_{s}|>\delta)\rightarrow 0$
as $t\rightarrow s$.

When $\alpha=2$, we obtain a scaled Brownian motion as a special case, i.e. $\Lm_{t}=\sqrt{2}\Bm_{t}$, so that
the difference $\Lm_{t}-\Lm_{s}$
follows a Gaussian distribution $\mathcal{N}(0,2(t-s))$.

\section{Tail-Index Estimation}\label{sec:alpha_estim}

In this study, we follow \citet{tzagkarakis2018compressive,pmlr-v97-simsekli19a}, and make use of the recent estimator proposed by \citet{mohammadi2015estimating}. 
\begin{theorem}[\citet{mohammadi2015estimating} Corollary 2.4]
Let $\{X_i\}_{i=1}^K$ be a collection of strictly stable random variables in $\mathbb{R}^d$ with tail-index $\alpha \in (0,2]$ and $K = K_1 \times K_2$.
Define $Y_i = \sum_{j=1}^{K_1} X_{j+(i-1)K_1} \>$ for $i \in \llbracket 1, K_2 \rrbracket$. Then, the estimator
\begin{align}
\label{eqn:alpha_estim}
\widehat{\phantom{a}\frac1{\alpha}\phantom{a}} \hspace{-4pt} \triangleq \hspace{-2pt} \frac1{\log K_1} \Bigl(\frac1{K_2 } \sum_{i=1}^{K_2} \log \|Y_i\|  - \frac1{K} \sum_{i=1}^K \log \|X_i\| \Bigr),
\end{align}
converges to $1/{\alpha}$ almost surely, as $K_2 \rightarrow \infty$.
\end{theorem}
As this estimator requires a hyperparameter $K_1$, at each tail-index estimation, we used several values for $K_1$ and we used the median of the estimators obtained with different values of $K_1$.
We provide the codes in \url{github.com/umutsimsekli/sgd_ht}, where the implementation details can be found. For the neural network experiments, we used the same setup as provided in the repository of \citet{pmlr-v97-simsekli19a}.



\section{Proofs of Main Results}\label{sec:proofs}

\subsection{Proof of Theorem~\ref{thm:main}}


\begin{proof}
The proof follows from Theorem~4.4.15 in \citet{buraczewski2016stochastic} which goes back to Theorem~1.1 in \citet{alsmeyer2012tail} and Theorem~6 in \citet{kesten1973random}. 
See also \citet{goldie1991implicit,bdp2015}.
We recall that we have the stochastic recursion:
\begin{equation}
x_{k}=M_{k}x_{k-1}+q_{k},
\end{equation}
where the sequence $(M_{k},q_{k})$ are i.i.d. 
distributed as $(M,q)$
and for each $k$, $(M_{k},q_{k})$ is independent of $x_{k-1}$.
To apply Theorem~4.4.15 in \citet{buraczewski2016stochastic}, 
it suffices to have the following conditions being satisfied:
\begin{enumerate}
    \item $M$ is invertible with probability 1.
    \item The matrix $M$ has a continuous Lebesgue density that is positive in a neighborhood of the identity matrix. 
    \item $\rho<0$ and $h(\alpha)=1$. 
    \item $\mathbb{P}(Mx + q = x) < 1$ for every $x$.
    \item $ \mathbb{E}\left[\|M\|^\alpha (\log^+\|M\| + \log^+\|M^{-1}\|)\right] < \infty$.
    \item $0<\mathbb{E}\|q\|^\alpha < \infty$.
\end{enumerate}
All the conditions are satisfied under our assumptions. In particular, Condition 1 and Condition 5 are proved in Lemma~\ref{M:finite}, 
and Condition 2 and Condition 4 follow from the fact that $M$ and $q$ have continuous
distributions. Condition 3 is part of the assumption of Theorem~\ref{thm:main}.
Finally, Condition 6 is satisfied by the definition of $q$
and by the Assumptions \textbf{(A1)}--\textbf{(A2)}.
\end{proof}

\subsection{Proof of Theorem~\ref{thm:Gaussian}}

\begin{proof}
To prove (i), according to the proof of Theorem~\ref{thm:main}, 
it suffices to show that if $\rho<0$, then there exists a unique
positive $\alpha$ such that $h(\alpha)=1$.
Note that if $\rho<0$, then 
by Lemma~\ref{h:property}, we have $h(0)=1$, $h'(0)=\rho<0$ and $h(s)$ is convex in $s$,
and moreover by Lemma~\ref{h:infinity}, 
we have $\liminf_{s\rightarrow\infty}h(s)>1$.
Therefore, there exists some $\alpha\in(0,\infty)$ such that $h(\alpha)=1$.
Finally, (ii) follows from Lemma~\ref{rho:h:iid}.
\end{proof}


\subsection{Proof of Theorem~\ref{thm:mono}}

\begin{proof}
We will split the proof of Theorem~\ref{thm:mono} into
two parts:

{(I) We will show that the tail-index $\alpha$ is strictly decreasing
in stepsize $\eta$ and variance $\sigma^{2}$ provided that $\alpha\geq 1$.

(II) We will show that the tail-index $\alpha$ is strictly increasing
in batch-size $b$ provided that $\alpha\geq 1$.

(III) We will show that the tail-index $\alpha$ is strictly decreasing
in dimension $d$.}


First, let us prove (I).
{Let $a := \eta\sigma^{2}> 0$ be given.}
Consider the tail-index $\alpha$ as a function of $a$, i.e. 
$$\alpha(a) := \min \{ s: h(a,s) = 1\}\,,$$
where $h(a,s)=h(s)$ with emphasis on dependence
on $a$.

By assumption, $\alpha(a)\geq 1$. The function $h(a,s)$ is convex function of $a$ (see Lemma~\ref{lemma-cvx-in-step} for $s\geq 1$ and a strictly convex function of $s$ for $s\geq 0$). Furthermore, it satisfies
$h(a,0) = 1$ for every $a\geq 0$ and $h(0,s) = 1$ for every $s\geq 0$. We consider the curve
$$ \mathcal{C}: = \left\{ (a,s) \in (0,\infty)\times [1,\infty] : h(a,s) = 1\right\}.$$
This is the set of the choice of $a$, which leads to a tail-index $s$ where $s\geq 1$. Since $h$ is smooth in both $a$ and $s$, we can represent $s$ as a smooth function of $a$, i.e. on the curve
 $$ h(a,s(a)) = 0\,, $$
where $s(a)$ is a smooth function of $a$. We will show that $s'(a) < 0$; i.e. if we increase $a$; the tail-index $s(a)$ will drop. 
Pick any $(a_*,s_*) \in \mathcal{C}$, it will satisfy $h(a_*,s_*)=1$. We have the following facts:
\begin{itemize}
    \item [$(i)$] The function $h(a,s) = 1$ for either $a=0$ or $s=0$. This is illustrated in Figure~\ref{fig:h2} with a blue marker.
    \item [$(ii)$] $h(a_*,s) < 1$ for $s<s_*$. This follows from the convexity of $h(a_*,s)$ function and the fact that $h(a_*,0)=1$, $h(a_*,s_*)=1$. From here, we see that the function $h(a_*,s)$ is increasing at $s=s_*$ and we have its derivative 
    
    $$ \frac{\partial h}{\partial s}(a_*,s_*) > 0.$$ 
    
    \item [$(iii)$] The function $h(a,s_*)$ is convex as a function of $a$ by Lemma~\ref{lemma-cvx-in-step}, it satisfies
    $h(0,s_*)=h(a_*,s_*)=1$. Therefore, by convexity $h(a,s_*)<1$ for a $\in (0,s_*)$; otherwise the function $h(a,s_*)$ would be a constant function. We have therefore necessarily. $$\frac{\partial h}{\partial a}(a_*,s_*)>0.$$
    By convexity of the function $h(a,s_*)$, we have also $h(a,s_*)\geq h(a_*,s_*) + \frac{\partial h}{\partial a}(a_*,s_*) (a-a_*) > h(a_*,s_*) =1$. Therefore,  $h(a,s_*)>1$ for $a>a_*$. Then, it also follows that $h(a,s)>1$ for $a>a_*$ and $s>s_*$ (otherwise if $h(a,s) \leq 1$, we get a contradiction because $h(0,s) = 1$, $h(a_*,s) > 1$ and $h(a,s) \leq 1$ is impossible due to convexity). This is illustrated in Figure~\ref{fig:h2} where we mark this region as a rectangular box where $h>1$. 
    \begin{figure}[h!]
    \centering
   \includegraphics[scale=0.5]{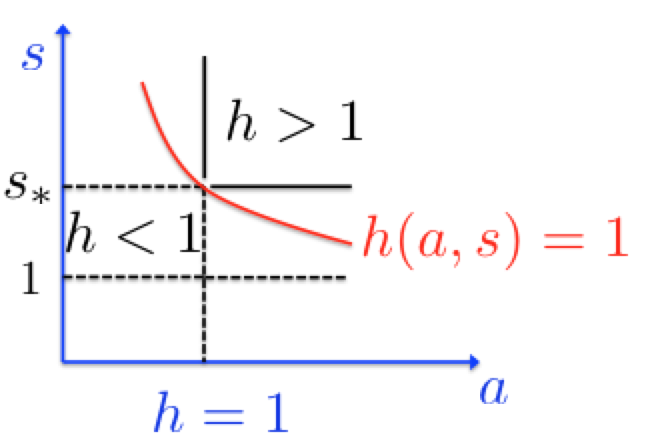}
    \caption{The curve $h(a,s)=1$ in the $(a,s)$ plane}
   \label{fig:h2} 
\end{figure}
    \item [$(iv)$] By similar arguments we can show that the function $h(a,s) < 1$ if $(s,a) \in (0,a_*)\times [1,s_*) $. Indeed, if $h(a,s)\geq 1$ for some $(s,a) \in [1,s_*)\times (0,a_*)$, this contradicts the fact that $h(0,s)=1$ and $h(a_*,s)<1$ proven in part $(ii)$. This is illustrated in Figure~\ref{fig:h2} where inside the rectangular box on the left-hand side, we have $h<1$.
\end{itemize}
Geometrically, we see from Figure~\ref{fig:h2} that the curve $s(a)$ as a function of $a$, is sandwiched between two rectangular boxes and has necessarily $s'(a) < 0$. This can also be directly obtained rigorously from the implicit function theorem; if we differentiate the implicit equation $ h(a,s(a)) = 0$ with respect to $a$, we obtain
        $$\frac{\partial h}{\partial a}(a_*,s_*) + \frac{\partial h}{\partial s}(a_*,s_*) s'(a_*) = 0\,.$$
   From parts $(ii)-(iii)$, we have $\frac{\partial h}{\partial a}(a_*,s_*)$ and $\frac{\partial h}{\partial s}(a_*,s_*)>0$. Therefore, we have 
   \begin{equation} s'(a_*) = - \frac{\frac{\partial h}{\partial a}(a_*,s_*)}{ \frac{\partial h}{\partial s}(a_*,s_*)} < 0\,, 
    \label{def-s-prime}
   \end{equation}
   which completes the proof for $s_*\geq 1$. 

Next, let us prove (II). With slight abuse of notation, we define the function $h(b,s)=h(s)$ to emphasize the dependence on $b$. We have
\begin{equation}
h(b,s)=\mathbb{E}\left\Vert\left(I-\frac{\eta}{b}\sum_{i=1}^{b}a_{i}a_{i}^{T}\right)e_{1}\right\Vert^{s}.\label{def-h-b-s}
\end{equation}
where we used Lemma~\ref{rho:h:iid}. When $s\geq 1$, the function $x\mapsto\Vert x\Vert^{s}$ is convex, 
and by Jensen's inequality, we get for any $b\geq 2$ and $b\in\mathbb{N}$,
\begin{align*}
h(b,s)&=\mathbb{E}\left\Vert\frac{1}{b}\sum_{i=1}^{b}\left(I-\frac{\eta}{b-1}\sum_{j\neq i}a_{j}a_{j}^{T}\right)e_{1}\right\Vert^{s}
\\
&\leq
\mathbb{E}\left[\frac{1}{b}\sum_{i=1}^{b}\left\Vert\left(I-\frac{\eta}{b-1}\sum_{j\neq i}a_{j}a_{j}^{T}\right)e_{1}\right\Vert^{s}\right]
\\
&=\frac{1}{b}\sum_{i=1}^{b}\mathbb{E}\left[\left\Vert\left(I-\frac{\eta}{b-1}\sum_{j\neq i}a_{j}a_{j}^{T}\right)e_{1}\right\Vert^{s}\right]
=h(b-1,s),
\end{align*}
where we used the fact that $a_{i}$ are i.i.d.
Indeed, from the condition for equality to hold in Jensen's inequality,
and the fact that $a_{i}$ are i.i.d. random,
the inequality above is a strict inequality.
Hence when $d\in\mathbb{N}$ for any $s\geq 1$, $h(b,s)$ is strictly decreasing in $b$.
By following the same argument as in the proof of (I),
we conclude that the tail-index $\alpha$ is strictly increasing in batch-size $b$.

Finally, let us prove (III). 
Let us show the tail-index $\alpha$
is strictly decreasing in dimension $d$.
Since $a_{i}$ are i.i.d. and $a_i \sim \mathcal{N}(0,\sigma^{2}I_{d})$, by Lemma~\ref{lem:simplified}, 
\begin{equation}
h(s)=\mathbb{E}\left[\left(1-\frac{2a}{b}X+\frac{a^{2}}{b^{2}}X^{2}+\frac{a^{2}}{b^{2}}XY\right)^{s/2}\right],
\end{equation}
where $X,Y$ are independent chi-square random variables
with degree of freedom $b$ and $d-1$ respectively.
Notice that $h(s)$ is strictly increasing in $d$ since the only dependence of $h(s)$
on $d$ is via $Y$, which is a chi-square distribution with degree of freedom $(d-1)$. By writing $Y=Z_{1}^{2}+\cdots+Z_{d-1}^{2}$, where $Z_{i}\sim N(0,1)$
i.i.d., it follows that $h(s)$ is strictly increasing in $d$. Hence, by similar argument as in (I), 
we conclude that $\alpha$ is strictly decreasing
in dimension $d$.
\end{proof}


\begin{remark}
When $d=1$ and $a_{i}$ are i.i.d. $N(0,\sigma^{2})$, 
we can provide an alternative proof that
the tail-index $\alpha$ is strictly increasing
in batch-size $b$.
It suffices to show that for any $s\geq 1$, 
$h(s)$ is strictly decreasing in the batch-size $b$.
By Lemma~\ref{lem:simplified} when $d=1$, 
\begin{equation}
h(b,s)=\mathbb{E}\left[\left(1-\frac{2\eta\sigma^{2}}{b}X+\frac{\eta^{2}\sigma^{4}}{b^{2}}X^{2}+\frac{\eta^{2}\sigma^{4}}{b^{2}}XY\right)^{s/2}\right],
\end{equation}
where $h(b,s)$ is as in (\ref{def-h-b-s}) and $X,Y$ are independent chi-square random variables
with degree of freedom $b$ and $d-1$ respectively.
When $d=1$, we have $Y\equiv 0$, and
\begin{equation}
h(b,s)=\mathbb{E}\left[\left(1-\frac{2\eta\sigma^{2}}{b}X+\frac{\eta^{2}\sigma^{4}}{b^{2}}X^{2}\right)^{s/2}\right]
=\mathbb{E}\left[\left|1-\frac{\eta\sigma^{2}}{b}X\right|^{s}\right].
\end{equation}
Since $X$ is a chi-square random variable
with degree of freedom $b$, we have
\begin{equation}
h(b,s)=\mathbb{E}\left[\left|1-\frac{\eta\sigma^{2}}{b}\sum_{i=1}^{b}Z_{i}\right|^{s}\right],
\end{equation}
where $Z_{i}$ are i.i.d. $N(0,1)$ random variables.
When $s\geq 1$, the function $x\mapsto |x|^{s}$ is convex,
and by Jensen's inequality, we get for any $b\geq 2$ and $b\in\mathbb{N}$
\begin{align*}
h(b,s)&=\mathbb{E}\left[\left|\frac{1}{b}\sum_{i=1}^{b}\left(1-\frac{\eta\sigma^{2}}{b-1}\sum_{j\neq i}Z_{j}\right)\right|^{s}\right]
\\
&\leq
\mathbb{E}\left[\frac{1}{b}\sum_{i=1}^{b}\left|1-\frac{\eta\sigma^{2}}{b-1}\sum_{j\neq i}Z_{j}\right|^{s}\right]
=\frac{1}{b}\sum_{i=1}^{b}\mathbb{E}\left[\left|1-\frac{\eta\sigma^{2}}{b-1}\sum_{j\neq i}Z_{j}\right|^{s}\right]=h(b-1,s),
\end{align*}
where we used the fact that $Z_{i}$ are i.i.d.
Indeed, from the condition for equality to hold in Jensen's inequality,
and the fact that $Z_{i}$ are i.i.d. $N(0,1)$ distributed,
the inequality above is a strict inequality.
Hence when $d=1$ for any $s\geq 1$, $h(b,s)$ is strictly decreasing in $b$.
\end{remark}


\subsection{Proof of Proposition~\ref{alpha:2}}
\begin{proof}
We first prove (i). 
When 
$\eta=\eta_{crit}=\frac{2b}{\sigma^{2}(d+b+1)}$, 
that is $\eta\sigma^{2}(d+b+1)=2b$, 
we can compute that
\begin{equation}
\rho
\leq\frac{1}{2}\log
\mathbb{E}\left[1- \frac{2\eta \sigma^2}{b} \sum_{i=1}^{b}z_{i1}^2 
+ \frac{\eta^2 \sigma^4}{b^{2}}\sum_{i=1}^{b}\sum_{j=1}^{b}(z_{i1}z_{j1}+\cdots+z_{id}z_{jd})z_{i1}z_{j1}\right]=0,
\end{equation}
where $z_{ij}$ are i.i.d. $N(0,1)$ random variables.
Note that since $$1- \frac{2\eta \sigma^2}{b} \sum_{i=1}^{b}z_{i1}^2 
+ \frac{\eta^2 \sigma^4}{b^{2}}\sum_{i=1}^{b}\sum_{j=1}^{b}(z_{i1}z_{j1}+\cdots+z_{id}z_{jd})z_{i1}z_{j1}$$ is
random, the inequality above is a strict inequality
from Jensen's inequality. Thus, when 
$\eta=\eta_{crit}$, 
i.e. $\eta\sigma^{2}(d+b+1)=2b$, $\rho<0$. 
By continuity, there exists some $\delta>0$
such that for any $2b<\eta\sigma^{2}(d+b+1)<2b+\delta$,
i.e. $\eta_{crit}<\eta<\eta_{max}$, where
$\eta_{max}:=\eta_{crit}+\frac{\delta}{\sigma^{2}(d+b+1)}$,
we have $\rho<0$. Moreover, when $\eta\sigma^{2}(d+b+1)>2b$,
i.e. $\eta>\eta_{crit}$,
we have
\begin{eqnarray*}
h(2)
&=&\mathbb{E}\left[\bigg(1- \frac{2\eta \sigma^2}{b} \sum_{i=1}^{b}z_{i1}^2 
+ \frac{\eta^2 \sigma^4}{b^{2}}\sum_{i=1}^{b}\sum_{j=1}^{b}(z_{i1}z_{j1}+\cdots+z_{id}z_{jd})z_{i1}z_{j1}\bigg)\right]\\
&=&1-2\eta\sigma^{2}+\frac{\eta^{2}\sigma^{4}}{b}(d+b+1)\geq 1,
\end{eqnarray*}
which implies that there exists some $0<\alpha<2$
such that $h(\alpha)=1$.

Finally, let us prove (ii) and (iii). 
When $\eta\sigma^{2}(d+b+1)\leq 2b$, i.e. $\eta\leq\eta_{crit}$, 
we have $h(2)\leq 1$, which implies that
$\alpha>2$. 
In particular, when $\eta\sigma^{2}(d+b+1)=2b$, i.e. $\eta=\eta_{crit}$,  
the tail-index $\alpha=2$.
\end{proof}

\subsection{Proof of Theorem~\ref{thm:moment} and
Corollary~\ref{cor:moment}}

\subsubsection{Proof of Theorem~\ref{thm:moment}}
\begin{proof}
We recall that 
\begin{equation}
x_{k}=M_{k}x_{k-1}+q_{k},
\end{equation}
which implies that
\begin{equation}
\Vert x_{k}\Vert
\leq\Vert M_{k}x_{k-1}\Vert
+\Vert q_{k}\Vert.
\end{equation}

(i) If the tail-index $\alpha\leq 1$,
then for any $0<p<\alpha$, 
we have $h(p)=\mathbb{E}\Vert M_{k}e_{1}\Vert^{p}<1$
and moreover by Lemma~\ref{lem:moment:ineq},
\begin{equation}
\Vert x_{k}\Vert^{p}
\leq\Vert M_{k}x_{k-1}\Vert^{p}
+\Vert q_{k}\Vert^{p}.
\end{equation}
{Due to spherical symmetry of the isotropic Gaussian distribution, the distribution of $\frac{\|M_k x\|}{\|x\|}$ does not depend on the choice of $x\in\mathbb{R}^d \backslash \{0\}$. Therefore, $\frac{\Vert M_{k}x_{k-1}\Vert}{\Vert x_{k-1}\Vert}$
and $\Vert x_{k-1}\Vert$ are independent, and $\frac{\Vert M_{k}x_{k-1}\Vert}{\Vert x_{k-1}\Vert}$ has the same distribution as $\Vert M_{k}e_{1}\Vert$, where
$e_{1}$ is the first basis vector.
It follows that
\begin{equation}
\mathbb{E}\Vert x_{k}\Vert^{p}
\leq
\mathbb{E}\Vert M_{k}e_{1}\Vert^{p}\mathbb{E}\Vert x_{k-1}\Vert^{p}
+\mathbb{E}\Vert q_{k}\Vert^{p},
\end{equation}
so that
\begin{equation}
\mathbb{E}\Vert x_{k}\Vert^{p}
\leq
h(p)\mathbb{E}\Vert x_{k-1}\Vert^{p}
+\mathbb{E}\Vert q_{1}\Vert^{p},
\end{equation}
where $h(p)\in(0,1)$.
By iterating over $k$, we get
\begin{equation}
\mathbb{E}\Vert x_{k}\Vert^{p}
\leq
(h(p))^{k}\mathbb{E}\Vert x_{0}\Vert^{p}
+\frac{1-(h(p))^{k}}{1-h(p)}\mathbb{E}\Vert q_{1}\Vert^{p}.
\end{equation}

(ii) If the tail-index $\alpha>1$, 
then for any $1<p<\alpha$, 
by Lemma~\ref{lem:moment:ineq}, 
for any $\epsilon>0$, we have
\begin{equation}
\Vert x_{k}\Vert^{p}
\leq
(1+\epsilon)\Vert M_{k}x_{k-1}\Vert^{p}
+\frac{(1+\epsilon)^{\frac{p}{p-1}}-(1+\epsilon)}{\left((1+\epsilon)^{\frac{1}{p-1}}-1\right)^{p}}\Vert q_{k}\Vert^{p},
\end{equation}
which (similar as in (i)) implies that
\begin{equation}
\mathbb{E}\Vert x_{k}\Vert^{p}
\leq
(1+\epsilon)
\mathbb{E}\Vert M_{k}e_{1}\Vert^{p}
\mathbb{E}\Vert x_{k-1}\Vert^{p}
+\frac{(1+\epsilon)^{\frac{p}{p-1}}-(1+\epsilon)}{\left((1+\epsilon)^{\frac{1}{p-1}}-1\right)^{p}}
\mathbb{E}\Vert q_{k}\Vert^{p},
\end{equation}}
so that
\begin{equation}
\mathbb{E}\Vert x_{k}\Vert^{p}
\leq
(1+\epsilon)h(p)
\mathbb{E}\Vert x_{k-1}\Vert^{p}
+\frac{(1+\epsilon)^{\frac{p}{p-1}}-(1+\epsilon)}{\left((1+\epsilon)^{\frac{1}{p-1}}-1\right)^{p}}
\mathbb{E}\Vert q_{1}\Vert^{p}.
\end{equation}
We choose $\epsilon>0$
so that $(1+\epsilon)h(p)<1$.
By iterating over $k$, we get
\begin{equation}
\mathbb{E}\Vert x_{k}\Vert^{p}
\leq
((1+\epsilon)h(p))^{k}\mathbb{E}\Vert x_{0}\Vert^{p}
+\frac{1-((1+\epsilon)h(p))^{k}}{1-(1+\epsilon)h(p)}
\frac{(1+\epsilon)^{\frac{p}{p-1}}-(1+\epsilon)}{\left((1+\epsilon)^{\frac{1}{p-1}}-1\right)^{p}}
\mathbb{E}\Vert q_{1}\Vert^{p}.
\end{equation}
The proof is complete.
\end{proof}


\begin{remark}
In general, there is
no closed-form expression 
for $\mathbb{E}\Vert q_{1}\Vert^{p}$ in Theorem~\ref{thm:moment}. 
We provide an upper bound as follows.
When $p>1$, by Jensen's inequality,
we can compute that
\begin{equation}
\mathbb{E}\Vert q_{1}\Vert^{p}
=\eta^{p}\mathbb{E}\left\Vert\frac{1}{b}\sum_{i=1}^{b}a_{i}y_{i}\right\Vert^{p}
\leq
\frac{\eta^{p}}{b}\sum_{i=1}^{b}
\mathbb{E}\left\Vert a_{i}y_{i}\right\Vert^{p}
=\eta^{p}\mathbb{E}\left[|y_{1}|^{p}\left\Vert a_{1}\right\Vert^{p}\right],
\end{equation}
and when $p\leq 1$, by Lemma~\ref{lem:moment:ineq}, 
we can compute that
\begin{equation}
\mathbb{E}\Vert q_{1}\Vert^{p}
=\frac{\eta^{p}}{b^{p}}\mathbb{E}\left\Vert\sum_{i=1}^{b}a_{i}y_{i}\right\Vert^{p}
\leq
\frac{\eta^{p}}{b^{p}}\mathbb{E}\left[\left(\sum_{i=1}^{b}
\left\Vert a_{i}y_{i}\right\Vert\right)^{p}\right]
\leq
\frac{\eta^{p}}{b^{p}}\sum_{i=1}^{b}
\mathbb{E}\left\Vert a_{i}y_{i}\right\Vert^{p}
=\eta^{p}\mathbb{E}\left[|y_{1}|^{p}\left\Vert a_{1}\right\Vert^{p}\right].
\end{equation}
\end{remark}

\subsubsection{Proof of Corollary~\ref{cor:moment}}
\begin{proof}
It follows from Theorem~\ref{thm:moment}
by letting $k\rightarrow\infty$ and applying
Fatou's lemma.
\end{proof}

\subsection{Proof of Theorem~\ref{thm:conv}, Corollary~\ref{cor:conv}, Proposition~\ref{prop:k:bound}
and Corollary~\ref{cor:clt}}

\subsubsection{Proof of Theorem~\ref{thm:conv}}
\begin{proof}
For any $\nu_{0},\tilde{\nu}_{0}\in\mathcal{P}_{p}(\mathbb{R}^{d})$, 
there exists a couple $x_{0}\sim\nu_{0}$ and $\tilde{x}_{0}\sim\tilde{\nu}_{0}$ independent
of $(M_{k},q_{k})_{k\in\mathbb{N}}$ and
$\mathcal{W}_{p}^{p}(\nu_{0},\tilde{\nu}_{0})=\mathbb{E}\Vert x_{0}-\tilde{x}_{0}\Vert^{p}$.
We define $x_{k}$ and $\tilde{x}_{k}$ starting from $x_{0}$ and $\tilde{x}_{0}$ respectively,
via the iterates
\begin{align}
&x_{k}=M_{k}x_{k-1}+q_{k},
\\
&\tilde{x}_{k}=M_{k}\tilde{x}_{k-1}+q_{k},
\end{align}
and let $\nu_{k}$ and $\tilde{\nu}_{k}$ denote
the probability laws of $x_{k}$ and $\tilde{x}_{k}$ respectively. 
For any $p<\alpha$, since $\mathbb{E}\Vert M_{k}\Vert^{\alpha}=1$
and $\mathbb{E}\Vert q_{k}\Vert^{\alpha}<\infty$, 
we have $\nu_{k},\tilde{\nu}_{k}\in\mathcal{P}_{p}(\mathbb{R}^{d})$
for any $k$. Moreover, we have
\begin{equation}\label{eqn:coupling}
x_{k}-\tilde{x}_{k}=M_{k}(x_{k-1}-\tilde{x}_{k-1}),
\end{equation}
{Due to spherical symmetry of the isotropic Gaussian distribution, the distribution of $\frac{\|M_k x\|}{\|x\|}$ does not depend on the choice of $x\in\mathbb{R}^d \backslash \{0\}$. 
Therefore, $\frac{\Vert M_{k}(x_{k-1}-\tilde{x}_{k-1})\Vert}{\Vert x_{k-1}-\tilde{x}_{k-1}\Vert}$
and $\Vert x_{k-1}-\tilde{x}_{k-1}\Vert$ are independent, and $\frac{\Vert M_{k}(x_{k-1}-\tilde{x}_{k-1})\Vert}{\Vert x_{k-1}-\tilde{x}_{k-1}\Vert}$ has the same distribution as $\Vert M_{k}e_{1}\Vert$, where
$e_{1}$ is the first basis vector.
It follows from (\ref{eqn:coupling}) that
\begin{eqnarray*}
\mathbb{E}\Vert x_{k}-\tilde{x}_{k}\Vert^{p}
\leq
\mathbb{E}\left[\Vert M_{k}(x_{k-1}-\tilde{x}_{k-1})\Vert^{p}\right]
&=&\mathbb{E}\left[\Vert M_{k}e_{1}\Vert^{p}\right]
\mathbb{E}\left[\Vert x_{k-1}-\tilde{x}_{k-1}\Vert^{p}\right]\\
&=&h(p)\mathbb{E}\left[\Vert x_{k-1}-\tilde{x}_{k-1}\Vert^{p}\right],
\end{eqnarray*}}
which by iterating implies that
\begin{equation}
\mathcal{W}_{p}^{p}(\nu_{k},\tilde{\nu}_{k})
\leq
\mathbb{E}\Vert x_{k}-\tilde{x}_{k}\Vert^{p}
\leq
(h(p))^{k}
\mathbb{E}\Vert x_{0}-\tilde{x}_{0}\Vert^{p}
=(h(p))^{k}
\mathcal{W}_{p}^{p}(\nu_{0},\tilde{\nu}_{0}).
\end{equation}
By letting $\tilde{\nu}_{0}=\nu_{\infty}$, the probability
law of the stationary distribution $x_{\infty}$, 
we conclude that
\begin{equation}
\mathcal{W}_{p}(\nu_{k},\nu_{\infty})
\leq
\left((h(p))^{1/q}\right)^{k}
\mathcal{W}_{p}(\nu_{0},\nu_{\infty}).
\end{equation}
Finally, notice that $1\leq p<\alpha$,
and therefore $h(p)<1$.
The proof is complete.
\end{proof}

\subsubsection{Proof of Corollary~\ref{cor:conv}}
\begin{proof}
When $\eta\sigma^{2}<\frac{2b}{d+b+1}$,
by Proposition~\ref{alpha:2},
the tail-index $\alpha>2$, by taking $p=2$,
and using $h(2)=1-2\eta\sigma^{2}+\frac{\eta^{2}\sigma^{4}}{b}(d+b+1)<1$ (see Proposition~\ref{alpha:2}), 
it follows from Theorem~\ref{thm:conv}
that
\begin{equation}
\mathcal{W}_{2}(\nu_{k},\nu_{\infty})
\leq
\left(1-2\eta\sigma^{2}
\left(1-\frac{\eta\sigma^{2}}{2b}(d+b+1)\right)\right)^{k/2}
\mathcal{W}_{2}(\nu_{0},\nu_{\infty}).
\end{equation}
\end{proof}


\begin{remark}
Consider the case $a_{i}$ are i.i.d. $\mathcal{N}(0,\sigma^{2}I_{d})$.
In Theorem~\ref{thm:moment}, Corollary~\ref{cor:moment} and Theorem~\ref{thm:conv},
the key quantity is $h(p)\in(0,1)$, where $p<\alpha$.
We recall that
\begin{equation}
h(p)=\mathbb{E}\left[\left(1-\frac{2a}{b}X+\frac{a^{2}}{b^{2}}X^{2}+\frac{a^{2}}{b^{2}}XY\right)^{p/2}\right],
\end{equation}
where $a=\eta\sigma^{2}$, $X,Y$ are independent chi-square random variables with degree
of freedom $b$ and $d-1$ respectively. 
The first-order approximation of $h(p)$ is given by
\begin{equation}
h(p)\sim
1+\frac{p}{2}\mathbb{E}\left[-\frac{2a}{b}X+\frac{a^{2}}{b^{2}}X^{2}+\frac{a^{2}}{b^{2}}XY\right]
=1+\frac{p}{2}\left[-2a+\frac{a^{2}}{b}(b+2)+\frac{a^{2}}{b}(d-1)\right]<1,
\end{equation}
provided that $a=\eta\sigma^{2}<\frac{2b}{d+b+1}$ which occurs if and only if $\alpha>2$.
In other words, when $\eta\sigma^{2}<\frac{2b}{d+b+1}$, $\alpha>2$ and
\begin{equation}
h(p)\sim
1-p\eta\sigma^{2}\left(1-\frac{\eta\sigma^{2}(b+d+1)}{2b}\right)<1.
\end{equation}
On the other hand, when $\eta\sigma^{2}\geq\frac{2b}{d+b+1}$, $p<\alpha\leq 2$,
and the second-order approximation of $h(p)$ is given by
\begin{align*}
h(p)&\sim
1+\frac{p}{2}\mathbb{E}\left[-\frac{2a}{b}X+\frac{a^{2}}{b^{2}}X^{2}+\frac{a^{2}}{b^{2}}XY\right]
+\frac{\frac{p}{2}(\frac{p}{2}-1)}{2}\mathbb{E}\left[\left(-\frac{2a}{b}X+\frac{a^{2}}{b^{2}}X^{2}+\frac{a^{2}}{b^{2}}XY\right)^{2}\right]
\\
&=1+qa\left(\frac{a(b+d+1)}{2b}-1\right)
-\frac{2-p}{8}\mathbb{E}\left[\left(-\frac{2a}{b}X+\frac{a^{2}}{b^{2}}X^{2}+\frac{a^{2}}{b^{2}}XY\right)^{2}\right],
\end{align*}
and 
for small $a=\eta\sigma^{2}$ and large $d$, 
\begin{equation}
\mathbb{E}\left[\left(-\frac{2a}{b}X+\frac{a^{2}}{b^{2}}X^{2}+\frac{a^{2}}{b^{2}}XY\right)^{2}\right]
\sim
\frac{4a^{2}}{b}(b+2)
+\frac{a^{4}}{b^{3}}(b+2)d^{2}
-\frac{4a^{3}}{b^{2}}(b+2)d,
\end{equation}
and therefore with $a=\eta\sigma^{2}$,
\begin{equation}
h(p)\sim
1-pa\left(\frac{-a(b+d+1)}{2b}+1
+\frac{(2-p)a(b+2)}{2qb}\left(1
+\frac{a^{2}}{4b^{2}}d^{2}
-\frac{a}{b}d\right)\right)<1,
\end{equation}
provided that $1\leq\frac{a(b+d+1)}{2b}<1
+\frac{(2-p)a(b+2)}{2qb}\left(1
+\frac{a^{2}}{4b^{2}}d^{2}
-\frac{a}{b}d\right)$. 
\end{remark}

\subsubsection{Proof of Proposition~\ref{prop:k:bound}}
\begin{proof}
First, we notice that it follows from Theorem~\ref{thm:main} that
$\mathbb{E}\Vert x_{\infty}\Vert^{\alpha}=\infty$. 
To see this, notice that
$\lim_{t\rightarrow\infty}t^{\alpha}\mathbb{P}(e_{1}^{T}x_{\infty}>t)
=e_{\alpha}(e_{1})$, where $e_{1}$ is the first basis
vector in $\mathbb{R}^{d}$, and 
$\mathbb{P}(\Vert x_{\infty}\Vert\geq t)
\geq\mathbb{P}(e_{1}^{T}x_{\infty}\geq t)$,
and thus 
\begin{equation}
\mathbb{E}\Vert x_{\infty}\Vert^{\alpha}
=\int_{0}^{\infty}t\mathbb{P}(\Vert x_{\infty}\Vert^{\alpha}\geq t)dt
=\int_{0}^{\infty}t\mathbb{P}(\Vert x_{\infty}\Vert\geq t^{1/\alpha})dt=\infty.
\end{equation}
By following the proof of Theorem~\ref{thm:moment} by
letting $q=\alpha$ in the proof, one can show the following.

(i) If the tail-index $\alpha\leq 1$, then we have
\begin{equation}
\mathbb{E}\Vert x_{\infty}\Vert^{\alpha}
\leq
\mathbb{E}\Vert x_{0}\Vert^{\alpha}+k\mathbb{E}\Vert q_{1}\Vert^{\alpha},
\end{equation}
which grows linearly in $k$. 

(ii) If the tail-index $\alpha>1$, then for any $\epsilon>0$, we have
\begin{equation}
\mathbb{E}\Vert x_{k}\Vert^{\alpha}
\leq
(1+\epsilon)^{k}\mathbb{E}\Vert x_{0}\Vert^{\alpha}
+\frac{(1+\epsilon)^{k}-1}{\epsilon}
\frac{(1+\epsilon)^{\frac{\alpha}{\alpha-1}}-(1+\epsilon)}{\left((1+\epsilon)^{\frac{1}{\alpha-1}}-1\right)^{\alpha}}
\mathbb{E}\Vert q_{1}\Vert^{\alpha}=O(k),
\end{equation}
which grows exponentially in $k$ for any fixed $\epsilon>0$.
By letting $\epsilon\rightarrow 0$, we have
\begin{align*}
\mathbb{E}\Vert x_{k}\Vert^{\alpha}
&=
(1+\epsilon)^{k}\mathbb{E}\Vert x_{0}\Vert^{\alpha}
+(1+O(\epsilon))\left((1+\epsilon)^{k}-1\right)\frac{(\alpha-1)^{\alpha-1}}{\epsilon^{\alpha}}
\mathbb{E}\Vert q_{1}\Vert^{\alpha}.
\end{align*}
Therefore, it holds for any sufficiently small $\epsilon>0$ that, 
\begin{align*}
\mathbb{E}\Vert x_{k}\Vert^{\alpha}
\leq
\frac{(1+\epsilon)^{k}}{\epsilon^{\alpha}}\left(\mathbb{E}\Vert x_{0}\Vert^{\alpha}
+(\alpha-1)^{\alpha-1}
\mathbb{E}\Vert q_{1}\Vert^{\alpha}\right).
\end{align*}
We can optimize $\frac{(1+\epsilon)^{k}}{\epsilon^{\alpha}}$ over the choice
of $\epsilon>0$, and by choosing $\epsilon=\frac{\alpha}{k-\alpha}$, which goes to zero
as $k$ goes to $\infty$, we have $\frac{(1+\epsilon)^{k}}{\epsilon^{\alpha}}=(1+\frac{\alpha}{k-\alpha})^{k}(\frac{k-\alpha}{\alpha})^{\alpha}=O(k^{\alpha})$, and hence
\begin{equation}
\mathbb{E}\Vert x_{k}\Vert^{\alpha}
=O(k^{\alpha}),
\end{equation}
which grows polynomially in $k$. 
The proof is complete.
\end{proof}

\subsubsection{Proof of Corollary~\ref{cor:clt}}
\begin{proof}
The result is obtained by a direct application of Theorem~1.15 in \citet{mirek2011heavy} to the recursions (\ref{eq-stoc-grad-2}) where it can be checked in a straightforward manner that the conditions for this theorem hold.
\end{proof}

\section{Supporting Lemmas}\label{sec:support}

In this section, we present a few supporting lemmas
that are used in the proofs of the main results
of the paper as well as the additional results
in the Appendix.

First, we recall that the iterates
are given by $x_{k}=M_{k}x_{k-1}+q_{k}$, 
where $(M_{k},q_{k})$ are i.i.d.
and $M_{k}$ is distributed as $I-\frac{\eta}{b}H$,
where $H=\sum_{i=1}^{b}a_{i}a_{i}^{T}$
and $q_{k}$ is distributed as $\frac{\eta}{b}\sum_{i=1}^{b}a_{i}y_{i}$,
where $a_{i}\sim\mathcal{N}(0,\sigma^{2}I_{d})$ and $y_{i}$ are i.i.d.
satisfying the Assumptions \textbf{(A1)}--\textbf{(A3)}.

We can compute $\rho$ and $h(s)$ as follows where $\rho$ and $h(s)$ are defined by (\ref{def-rho}) and (\ref{def-hs}).

\begin{lemma}\label{rho:h:iid}
Under Assumptions \textbf{(A1)}--\textbf{(A3)}, $\rho$ can be characterized as:
\begin{align}
&\rho=\mathbb{E}\left[\log\left\Vert\left(I-\frac{\eta}{b} H\right)e_{1}\right\Vert\right], \label{eq-rho-to-prove}
\end{align}
and $h(s)$ can be characterized as:
\begin{align}
h(s)=\mathbb{E}\left[\left\Vert\left(I-\frac{\eta}{b}H\right)e_{1}\right\Vert^{s}\right], \label{eq-h-to-prove}
\end{align}
provided that $\rho<0$. Furthermore, we have
\begin{equation} \hat{\rho} = \mathbb{E}\log\left\Vert\left(I-\frac{\eta}{b} H\right)e_{1}\right\Vert, \quad \hat {h}(s) = \mathbb{E}\left[\left\Vert\left(I-\frac{\eta}{b}H\right)e_{1}\right\Vert^{s}\right]
\label{eq-rho-hatrho}
\end{equation}
where $\hat{\rho}$ and $\hat h(s)$ are defined in \eqref{ineq-h-s}.
\end{lemma}

\begin{proof}
It is known that the Lyapunov exponent defined in (\ref{def-rho}) admits the alternative representation
\begin{equation}  \rho := \lim_{k\to\infty} \frac{1}{k}\log\|\tilde{x}_k\|\,,
\label{def-rho-alternative}
\end{equation}
where $\tilde{x}_k := \Pi_k \tilde{x}_0$ with $\Pi_k := M_k M_{k-1} \dots M_1$ and $\tilde{x}_0 := x_0$ (see Equation~(2) in \citet{newman1986distribution}). We will compute the limit on the right-hand side of (\ref{def-rho-alternative}). First, we observe that due to spherical symmetry of the isotropic Gaussian distribution, the distribution of $\frac{\|M_k x\|}{\|x\|}$ does not depend on the choice of $x\in\mathbb{R}^d \backslash \{0\}$  and is i.i.d. over $k$ with the same distribution as $\|M e_{1}\|$ where we chose $x=e_1$. This observation would directly imply the equality \eqref{eq-rho-hatrho}. In addition, $$\frac{1}{k}\log\|\tilde{x}_k\| - \frac{1}{k} \log \|\tilde{x}_0\| = \frac{1}{k} \sum_{i=1}^{k}\log\frac{ \|\tilde{x}_{i}\|}{\|\tilde{x}_{i-1}\|} =\frac{1}{k} \sum_{i=1}^{k} \log\frac{ \|M_i\tilde{x}_{i-1}\|}{\|\tilde{x}_{i-1}\|}$$ is an average of i.i.d.\ random variables and by the law of large numbers we obtain 
\begin{equation*}
\rho=\lim_{k\rightarrow\infty} \frac{1}{k}\log\|\tilde{x}_k\| = \mathbb{E}\left[\log\left\|\left(I-\frac{\eta}{b}H \right) e_1\right\|\right]. 
\end{equation*}
From (\ref{def-rho-alternative}), we conclude that this proves (\ref{eq-rho-to-prove}).

It remains to prove (\ref{eq-h-to-prove}). We consider the function $$  \tilde{h}(s) = \lim_{k\to\infty} \left( \mathbb{E} \frac{\|\tilde{x}_k\|^s}{\|\tilde{x}_0\|^s}\right)^{1/k},$$ 
where the initial point $\tilde{x}_0 = x_0$ is deterministic. In the rest of the proof, we will show that for $\rho<0$, $h(s) = \tilde{h}(s)$ where $h(s)$ is given by (\ref{def-hs}) and $\tilde{h}(s)$ is equal to the right-hand side of (\ref{eq-h-to-prove}); our proof is inspired by the approach of \citet{newman1986distribution}. 
We will first compute $\tilde{h}(s)$ and show that it is equal to the right-hand side of (\ref{eq-h-to-prove}). Note that we can write
$$\frac{\|\tilde{x}_k\|^s}{\|\tilde{x}_0\|^s} = \prod_{i=1}^{k} \frac{ \|M_i\tilde{x}_{i-1}\|^s}{\|\tilde{x}_{i-1}\|^s}.$$
This is a product of i.i.d. random variables with the same distribution as that of $\|Me_1\|^s$ due to the spherical symmetry of the input $a_i$. Therefore, we can write 
\begin{align} 
\tilde{h}(s)  = \lim_{k\to\infty} \left( \mathbb{E} \frac{\|\tilde{x}_k\|^s}{\|\tilde{x}_0\|^s}\right)^{1/k} = \lim_{k\to\infty} \left( \mathbb{E} \prod_{i=1}^k \left\|M_i e_1\right\|^s \right)^{1/k}
= \mathbb{E}\left[\left\Vert M e_{1}\right\Vert^{s}\right] = \mathbb{E}\left[\left\Vert\left(I-\frac{\eta}{b}H\right)e_{1}\right\Vert^{s}\right]\,,
\end{align} 
where we used the fact that $M_i e_1$ are i.i.d. over $i$.
It remains to show that $h(s) = \tilde{h}(s)$ for $\rho<0$.
Note that 
$ \frac{\|\tilde{x}_k\|^s}{\|\tilde{x}_0\|^s} \leq  \| \Pi_k\|^s $, 
and therefore from the definition of $h(s)$ and $\tilde{h}(s)$, we have immediately
   \begin{equation}   h(s) \geq \tilde{h}(s)
   \label{ineq-hs}
   \end{equation}
for any $s>0$. We will show that $h(s) \leq \tilde{h}(s)$ when $\rho<0$.
We assume $\rho<0$. Then, Theorem~\ref{thm:main} is applicable and there exists a stationary distribution $x_\infty$ with a tail-index $\alpha$ such that $h(\alpha)=1$. We will show that $\tilde{h}(\alpha)=1$. First, the tail density admits the characterization (\ref{eq-heavy-tail}), and therefore $x_\infty \in L_s$ for $s<\alpha$, i.e. the $s$-th moment of $x_\infty$ is finite. Similarly due to (\ref{eq-heavy-tail}), $x_\infty \notin L_s$ for $s>\alpha$. Since $h(\alpha)=1$, it follows from (\ref{ineq-hs}) that we have $\tilde{h}(\alpha) \leq 1$. However if  $\tilde{h}(\alpha) < 1$, then by the continuity of the $\tilde{h}$ function there exists $\varepsilon$ such that $h(s) < 1$ for every $s\in (\alpha-\varepsilon, \alpha+ \varepsilon) \subset (0,1)$. From the definition of $\tilde{h}(s)$ then this would imply that $\mathbb{E}(\|x_k\|^{s}) \to 0$ for every $s\in (\alpha-\varepsilon, \alpha+ \varepsilon)$. On the other hand, by following a similar argument to the proof technique of Corollary~\ref{cor:moment}, it can be shown that the $s$-th moment of $x_\infty$ has to be bounded,\footnote{Note that the proof of Corollary~\ref{cor:moment} establishes first that $x_\infty$ has a bounded $s$-th moment provided that $\tilde{h}(s)=\mathbb{E}\left[\left\Vert M e_{1}\right\Vert^{s}\right]  <1$ and then cites Lemma~\ref{rho:h:iid} regarding the equivalence $h(s)=\tilde{h}(s)$.} 
which would be a contradiction with the fact that $x_\infty \not\in L_s$ for $s>\alpha$.
Therefore, $\tilde{h}(\alpha)\geq 1$. Since $h(\alpha)=1$, (\ref{ineq-hs}) leads to
\begin{equation}
    h(\alpha) = \tilde{h}(\alpha)=1.
    \label{eq-h-vs-hs}
\end{equation}
We observe that the function $h$ is homogeneous in the sense that if the iterations matrices $M_i$ are replaced by $cM_i$ where $c>0$ is a real scalar, $h(s)$ will be replaced by $h_c(s):=c^s h(s)$. In other words, the function
\begin{equation}
h_c(s) := \lim\nolimits_{k\to\infty}\left(\mathbb{E}\left\| (cM_k) (cM_{k-1})\dots (cM_1)\right\|^s\right)^{1/k}
\end{equation}
clearly satisfies $h_c(s) = c^s h(s)$ by definition. A similar homogeneity property holds for $\tilde{h}(s)$: If the iterations matrices $M_i$ are replaced by $cM_i$, then $\tilde{h}(s)$ will be replaced by $\tilde{h}_c(s):= c^s \tilde{h}(s)$. 
We will show that this homogeneity property combined with the fact that $h(\alpha) = \tilde{h}(\alpha)=1$ will force $h(s) = \tilde{h}(s)$ for any $s>0$. For this purpose, given $s>0$, we choose $c = 1/\sqrt[s]{h(s)}$. Then, by considering input matrix $cM_i$ instead of $M_i$ and by following a similar argument which led to the identity (\ref{eq-h-vs-hs}), we can show that $h_c(s) = c^s h(s) = 1$. Therefore, $\tilde{h}_c(s)= \tilde{h}_c(s) = 1$. This implies directly $\tilde{h}(s)=h(s)$.

\end{proof}

Next, we show the following property for the function $h$.
\begin{lemma}\label{h:property}
We have $h(0)=1$, $h'(0)=\rho$ and $h(s)$ is strictly convex in $s$. 
\end{lemma}

\begin{proof}
By the expression of $h(s)$ from Lemma~\ref{rho:h:iid}, it is easy to check that $h(0)=1$.
Moreover, we can compute that
\begin{equation}
h'(s)=\mathbb{E}\left[\log\left(\left\Vert\left(I-\frac{\eta}{b} H\right)e_{1}\right\Vert\right)\left\Vert\left(I-\frac{\eta}{b} H\right)e_{1}\right\Vert^{s}\right],
\end{equation}
and thus
$h'(0)=\rho$.
Moreover, we can compute that
\begin{equation}
h''(s)=\mathbb{E}\left[\left(\log\left(\left\Vert\left(I-\frac{\eta}{b} H\right)e_{1}\right\Vert\right)\right)^{2}\left\Vert\left(I-\frac{\eta}{b} H\right)e_{1}\right\Vert^{s}\right]>0,
\end{equation}
which implies that $h(s)$ is strictly convex in $s$. 
\end{proof}


In the next result, we show that 
$\liminf_{s\rightarrow\infty}h(s)>1$.
This property, together with Lemma~\ref{h:property}
implies that if $\rho<0$, then there exists
some $\alpha\in(0,\infty)$ such that $h(\alpha)=1$.
Indeed, in the proof of Lemma~\ref{h:infinity},
we will show that $\liminf_{s\rightarrow\infty}h(s)=\infty$.

\begin{lemma}\label{h:infinity}
We have $\liminf_{s\rightarrow\infty}h(s)>1$.
\end{lemma}

\begin{proof}
We recall from Lemma~\ref{rho:h:iid} that
\begin{equation}
h(s)=\mathbb{E}\left\Vert\left(I-\frac{\eta}{b} H\right)e_{1}\right\Vert^{s},
\end{equation}
where $e_{1}$ is the first basis vector in $\mathbb{R}^{d}$ 
and $H=\sum_{i=1}^{b}a_{i}a_{i}^{T}$,
and $a_{i}=(a_{i1},\ldots,a_{id})$ are i.i.d.
distributed as $\mathcal{N}(0,\sigma^{2}I_{d})$.
We can compute that
\begin{align*}
\mathbb{E}\left\Vert\left(I-\frac{\eta}{b} H\right)e_{1}\right\Vert^{s}
&=\mathbb{E}\left(\left\Vert\left(I-\frac{\eta}{b} H\right)e_{1}\right\Vert^{2}\right)^{s/2}
\\
&=\mathbb{E}\left[\left(e_{1}^{T}\left(I-\frac{\eta}{b}\sum_{i=1}^{b}a_{i}a_{i}^{T}\right)\left(I-\frac{\eta}{b}\sum_{i=1}^{b}a_{i}a_{i}^{T}\right)e_{1}\right)^{s/2}\right]
\\
&=\mathbb{E}\left[
\left(1-\frac{2\eta}{b}e_{1}^{T}\sum_{i=1}^{b}a_{i}a_{i}^{T}e_{1}
+\frac{\eta^{2}}{b^{2}}e_{1}^{T}\sum_{i=1}^{b}a_{i}a_{i}^{T}\sum_{i=1}^{b}a_{i}a_{i}^{T}e_{1}\right)^{s/2}\right]
\\
&=\mathbb{E}\left[\left(1- \frac{2\eta}{b} \sum_{i=1}^{b}a_{i1}^2 
+ \frac{\eta^2 }{b^{2}}\sum_{i=1}^{b}\sum_{j=1}^{b}(a_{i1}a_{j1}+\cdots+a_{id}a_{jd})a_{i1}a_{j1}\right)^{s/2}\right]
\\
&=\mathbb{E}\left[\left(\left(1- \frac{\eta}{b} \sum_{i=1}^{b}a_{i1}^2\right)^{2} 
+\frac{\eta^{2}}{b^{2}}\sum_{i=1}^{b}\sum_{j=1}^{b}(a_{i2}a_{j2}+\cdots+a_{id}a_{jd})a_{i1}a_{j1}\right)^{s/2}\right]
\\
&\geq
\mathbb{E}\left[2^{s/2}1_{\frac{\eta^{2}}{b^{2}}\sum_{i=1}^{b}\sum_{j=1}^{b}(a_{i2}a_{j2}+\cdots+a_{id}a_{jd})a_{i1}a_{j1}\geq 2}\right]
\\
&=2^{s/2}\mathbb{P}\left(\frac{\eta^{2}}{b^{2}}\sum_{i=1}^{b}\sum_{j=1}^{b}(a_{i2}a_{j2}+\cdots+a_{id}a_{jd})a_{i1}a_{j1}\geq 2\right)\rightarrow\infty,
\end{align*}
as $s\rightarrow\infty$.
\end{proof}

Next, we provide alternative formulas
for $h(s)$ and $\rho$ for the Gaussian data which is used
for some technical proofs.

\begin{lemma}\label{lem:simplified}
For any $s>0$, 
\begin{equation*}
h(s)
=\mathbb{E}\left[\left(\left(1-\frac{\eta\sigma^{2}}{b}X\right)^{2}
+\frac{\eta^{2}\sigma^{4}}{b^{2}}XY\right)^{s/2}\right],
\end{equation*}
and
\begin{align*}
\rho=\frac{1}{2}
\mathbb{E}\left[\log\left(\left(1-\frac{\eta\sigma^{2}}{b}X\right)^{2}+\frac{\eta^{2}\sigma^{4}}{b^{2}}XY\right)\right],
\end{align*}
where $X,Y$ are independent and $X$ is chi-square random variable
with degree of freedom $b$ and $Y$ is a chi-square random variable
with degree of freedom $(d-1)$.
\end{lemma}

\begin{proof}
We can compute that
\begin{align*}
h(s)&=\mathbb{E}\left[\left(1- \frac{2\eta \sigma^2}{b} \sum_{i=1}^{b}z_{i1}^2 
+ \frac{\eta^2 \sigma^4}{b^{2}}\sum_{i=1}^{b}\sum_{j=1}^{b}(z_{i1}z_{j1}+\cdots+z_{id}z_{jd})z_{i1}z_{j1}\right)^{s/2}\right]
\\
&=\mathbb{E}\left[\left(1- \frac{2\eta \sigma^2}{b} \sum_{i=1}^{b}z_{i1}^2 
+ \frac{\eta^2 \sigma^4}{b^{2}}\sum_{i=1}^{b}\sum_{j=1}^{b}\left(z_{i1}^{2}z_{j1}^{2}+z_{i1}z_{j1}\sum_{k=2}^{d}z_{ik}z_{jk}\right)\right)^{s/2}\right]
\\
&=\mathbb{E}\left[\left(1- \frac{2\eta \sigma^2}{b} \sum_{i=1}^{b}z_{i1}^2 
+ \frac{\eta^2 \sigma^4}{b^{2}}\left(\sum_{i=1}^{b}z_{i1}^{2}\right)^{2}
+\frac{\eta^{2}\sigma^{4}}{b^{2}}\sum_{k=2}^{d}\left(\sum_{i=1}^{b}z_{i1}z_{ik}\right)^{2}\right)^{s/2}\right]
\\
&=\mathbb{E}\left[\left(\left(1- \frac{\eta \sigma^2}{b} \sum_{i=1}^{b}z_{i1}^2\right)^{2}
+\frac{\eta^{2}\sigma^{4}}{b^{2}}\sum_{k=2}^{d}\left(\sum_{i=1}^{b}z_{i1}z_{ik}\right)^{2}\right)^{s/2}\right],
\end{align*}
where $z_{ij}$ are i.i.d. $N(0,1)$ random variables.
Note that conditional on $z_{i1}$, $1\leq i\leq b$, 
\begin{equation}
\sum_{i=1}^{b}z_{i1}z_{ik}\sim\mathcal{N}\left(0,\sum_{i=1}^{b}z_{i1}^{2}\right),
\end{equation}
are i.i.d. for $k=2,\ldots,d$.
Therefore, we have
\begin{align*}
h(s)&=\mathbb{E}\left[\left(\left(1- \frac{\eta \sigma^2}{b} \sum_{i=1}^{b}z_{i1}^2\right)^{2}
+\frac{\eta^{2}\sigma^{4}}{b^{2}}\sum_{k=2}^{d}\left(\sum_{i=1}^{b}z_{i1}z_{ik}\right)^{2}\right)^{s/2}\right]
\\
&=\mathbb{E}\left[\left(\left(1- \frac{\eta \sigma^2}{b} \sum_{i=1}^{b}z_{i1}^2\right)^{2}
+\frac{\eta^{2}\sigma^{4}}{b^{2}}\sum_{i=1}^{b}z_{i1}^{2}\sum_{k=2}^{d}x_{k}^{2}\right)^{s/2}\right],
\end{align*}
where $x_{k}$ are i.i.d. $N(0,1)$ independent of $z_{i1}$, $i=1,\ldots,b$.
Hence, we have
\begin{align*}
h(s)&=\mathbb{E}\left[\left(\left(1- \frac{\eta \sigma^2}{b} \sum_{i=1}^{b}z_{i1}^2\right)^{2}
+\frac{\eta^{2}\sigma^{4}}{b^{2}}\sum_{i=1}^{b}z_{i1}^{2}\sum_{k=2}^{d}x_{k}^{2}\right)^{s/2}\right]
\\
&=\mathbb{E}\left[\left(\left(1-\frac{\eta\sigma^{2}}{b}X\right)^{2}+\frac{\eta^{2}\sigma^{4}}{b^{2}}XY\right)^{s/2}\right],
\end{align*}
where $X,Y$ are independent and $X$ is chi-square random variable
with degree of freedom $b$ and $Y$ is a chi-square random variable
with degree of freedom $(d-1)$.

Similarly, we can compute that
\begin{align*}
\rho
&=\frac{1}{2}\mathbb{E}\left[\log\left[\left(1- \frac{\eta \sigma^2}{b} \sum_{i=1}^{b}z_{i1}^2\right)^{2} 
+ \frac{\eta^2 \sigma^4}{b^{2}}\sum_{i=1}^{b}\sum_{j=1}^{b}z_{i1}z_{j1}\sum_{k=2}^{d}z_{ik}z_{jk}\right]\right]
\\
&=\frac{1}{2}\mathbb{E}\left[\log\left[\left(1- \frac{\eta \sigma^2}{b} \sum_{i=1}^{b}z_{i1}^2\right)^{2} 
+ \frac{\eta^2 \sigma^4}{b^{2}}\sum_{k=2}^{d}\left(\sum_{i=1}^{b}z_{i1}z_{ik}\right)^{2}\right]\right]
\\
&=\frac{1}{2}
\mathbb{E}\left[\log\left(\left(1-\frac{\eta\sigma^{2}}{b}X\right)^{2}+\frac{\eta^{2}\sigma^{4}}{b^{2}}XY\right)\right],
\end{align*}
where $X,Y$ are independent and $X$ is chi-square random variable
with degree of freedom $b$ and $Y$ is a chi-square random variable
with degree of freedom $(d-1)$. The proof is complete.
\end{proof}


In the next result, we show that the inverse of $M$
exists with probability $1$, and provide
an upper bound result, which will be used
to prove Lemma~\ref{M:finite}.

\begin{lemma}\label{lemma-inv-norm-estimate} 
Let $a_i$ satisfy Assumption \textbf{(A1)}. Then,
$M^{-1}$ exists with probability $1$.
Moreover, we have
$$\mathbb{E} \left[ \left( \log^+\|M^{-1}\|\right)^2 \right] \leq 8.$$
\end{lemma}

\begin{proof} 
Note that $M$ is a continuous random matrix, by the assumption on the distribution of $a_i$. 
Therefore,
\begin{equation}
\mathbb{P}\left(\text{$M^{-1}$ does not exist}\right)
=\mathbb{P}(\det M=0)=0.
\end{equation}
Note that the singular values of $M^{-1}$ are of the form $|1-\frac{\eta}{b} \sigma_H|^{-1}$ where $\sigma_H$ is a singular value of $H$ and we have
\begin{equation}
\left(\log^+\left\|M^{-1}\right\|\right)^2 = \begin{cases}
0 & \mbox{if} \quad  \frac{\eta}{b} H \succ 2I\,, \\
\left(\left\|(I-\frac{\eta}{b} H)^{-1}\right\|\right)^2 & \mbox{if} \quad  0 \preceq \frac{\eta}{b} H \preceq 2I\,.
\end{cases} 
\label{eq-logplus}
\end{equation}
We consider two cases $0\preceq \frac{\eta}{b} H \preceq I$ and $I\preceq \frac{\eta}{b} H \preceq 2I$. We compute the conditional expectations for each case:
\begin{align} 
\mathbb{E} \left[ \left(\log^+\left\|M^{-1}\right\|\right)^2  ~\big|~ 0\preceq \frac{\eta}{b} H \preceq I 
\right] &= \mathbb{E} \left[
\left(\log\left\|\left(I-\frac{\eta}{b} H\right)^{-1}\right\|\right)^2 ~\Big|~ 0\preceq \frac{\eta}{b} H \prec I 
\right] \\
&\leq \mathbb{E} \left[ \left(2\frac{\eta}{b} \|H\|\right)^2 ~\Big|~ 0\preceq \frac{\eta}{b} H \preceq I 
\right]\\
&\leq 4\,, \label{ineq-minv-est1}
\end{align}
where in the first inequality we used the fact that 
\begin{equation} 
\log (I-X)^{-1} \preceq 2X
\label{ineq-psd}
\end{equation} for a symmetric positive semi-definite matrix $X$ satisfying $0\preceq X  \prec I$ (the proof of this fact is analogous to the proof of the scalar inequality $\log(\frac{1}{1-x})\leq 2x$ for $0\leq x<1$). By a similar computation,   
\begin{align*} 
&\mathbb{E} \bigg[  (\log^+\|M^{-1} \|)^2  ~\big|~ I\preceq \frac{\eta}{b} H \preceq 2I \bigg] 
\\
&= \mathbb{E} \left[
\log\left\|\left(I-\frac{\eta}{b} H\right)^{-1}\right\| ~\big|~ I \preceq  \frac{\eta}{b} H  \prec 2I
\right]
\\
&= \mathbb{E}\bigg[
\log^2 \left\| \left(\frac{\eta}{b} H\right)^{-1} \left[I-\left(\frac{\eta}{b} H\right)^{-1}\right]^{-1}
\right\| ~\big|~ I \preceq \frac{\eta}{b} H  \prec 2I   \bigg]\\
&\leq \mathbb{E} \bigg[
\log^2 \bigg( \left\| \left(\frac{\eta}{b} H\right)^{-1}\right\| \cdot \left\|\left[I-\left(\frac{\eta}{b} H\right)^{-1}\right]^{-1}\right\|\bigg) ~\big|~ I \preceq \frac{\eta}{b} H  \prec 2I \bigg]
 \\ 
&\leq \mathbb{E} \bigg[
\log^2 \bigg( \left\|\left[I-\left(\frac{\eta}{b} H\right)^{-1}\right]^{-1}\right\| \bigg) ~\big|~ I \preceq \frac{\eta}{b} H  \prec 2I \bigg]
\\
&=\mathbb{E} \bigg[
\log^2 \bigg( \left\|\left[I-\left(\frac{\eta}{b} H\right)^{-1}\right]^{-1}\right\|\bigg) ~\big|~ \frac{1}{2} I \preceq \left(\frac{\eta}{b} H\right)^{-1}  \prec I \bigg]\,,
\end{align*}
where in the last inequality we used the fact that $(\frac{\eta}{b} H)^{-1}\preceq I$ for $I \preceq \frac{\eta}{b} H  \prec 2I$. If we apply the inequality (\ref{ineq-psd}) to the last inequality for the choice of $X=(\frac{\eta}{b} H)^{-1}$, we obtain
\begin{align} 
&\mathbb{E}
\left[ \log^2\left\|\left[I-\left(\frac{\eta}{b} H\right)^{-1}\right]^{-1}\right\| ~\Big|~ \frac{1}{2} I \preceq \left(\frac{\eta}{b} H\right)^{-1}  \prec I 
\right] \nonumber
\\
&\leq  \mathbb{E} \left[\left\|2\left(\frac{\eta}{b} H\right)^{-1}\right\|^2  ~\Big|~ \frac{1}{2} I \preceq \left(\frac{\eta}{b} H\right)^{-1}  \prec I \right]
\leq 4\,. \label{ineq-minv-est2}
\end{align} 
Combining (\ref{ineq-minv-est1}) and (\ref{ineq-minv-est2}), it follows from (\ref{eq-logplus}) that $\mathbb{E} \log^+ \|M^{-1}\| \leq 8$.
\end{proof}

In the next result, we show
that a certain expected value
that involves the moments and logarithm 
of $\Vert M\Vert$, and logarithm of 
$\Vert M^{-1}\Vert$ is finite,
which is used in the proof of Theorem~\ref{thm:main}.

\begin{lemma}\label{M:finite}
Let $a_i$ satisfy Assumption \textbf{(A1)}. Then,
\begin{equation*} 
\mathbb{E}\left[\|M\|^\alpha \left(\log^+\|M\| + \log^+\|M^{-1}\|\right)\right]<\infty\,.
\end{equation*}
\end{lemma}

\begin{proof}
Note that $M=I-\frac{\eta}{b}H$, where $H=\sum_{i}^{b}a_{i}a_{i}^{T}$ in distribution. 
Therefore for any $s>0$,
\begin{equation}
\mathbb{E}[\Vert M\Vert^{s}]
= \mathbb{E} \left[\left\Vert I-\frac{\eta}{b}\sum_{i=1}^b a_i a_i^T \right\Vert^{s}\right]
\leq \mathbb{E} \left[\left(1+ \frac{\eta}{b} \sum_{i=1}^b \|a_i\|^2\right)^{s}\right]<\infty\,,
\end{equation}
since all the moments of $a_i$ are finite by the Assumption \textbf{(A1)}. 
This implies that
\begin{equation*} 
\mathbb{E}\left[\|M\|^\alpha \left(\log^+\|M\| \right)\right]<\infty\,.
\end{equation*}
By Cauchy-Schwarz inequality,
\begin{equation*} 
\mathbb{E}\left[\|M\|^\alpha \left(\log^+\|M^{-1}\|\right)\right]
\leq
\left(\mathbb{E}\left[\|M\|^{2\alpha}\right]
\mathbb{E}\left[\left(\log^+\|M^{-1}\|\right)^{2}\right]\right)^{1/2}<\infty,
\end{equation*}
where we used Lemma~\ref{lemma-inv-norm-estimate}.
\end{proof}


In the next result, we show
a convexity result, which is used
in the proof of Theorem~\ref{thm:mono}
to show that the tail-index $\alpha$
is strictly decreasing in stepsize
$\eta$ and variance $\sigma^{2}$.

\begin{lemma}\label{lemma-cvx-in-step} 
{For any given positive semi-definite symmetric matrix $H$ fixed, the function $F_H:[0,\infty) \to \mathbb{R}$ defined as
$$ F_H(a) :=\left\| \left(I - a H\right)e_{1}\right\|^s $$
is convex for $s\geq 1$. It follows that for given $b$ and $d$ with $\tilde{H} := \frac{1}{b}\sum_{i=1}^b a_ia_i^T$, the function
\begin{equation}
    h(a,s) :=  \mathbb{E} \left[F_{\tilde{H}}(a)\right] =\mathbb{E}\left\| \left(I - a \tilde{H}\right)e_{1} \right\|^s
\end{equation}
is a convex function of $a$ for a fixed $s\geq 1$.} 
\end{lemma}

\begin{proof} 
We consider the case $s\geq 1$ and consider the function
$$ G_H(a) :=\left\| \left(I - a H\right)e_{1} \right\|, $$
and show that it is convex for $H\succeq 0$ and it is strongly convex for $H\succ 0$ over the interval $[0,\infty)$. Let $a_1,a_2 \in [0,\infty)$ be different points, i.e. $a_1 \neq a_2$.
It follows from the subadditivity of the norm that
\begin{align*} 
G_H\left(\frac{a_1 + a_2}{2}\right)
&= \left\| \left(I - \frac{a_1+a_2}{2} H\right)e_{1} \right\|
\\
&\leq \left\| \left(\frac{I}{2} - \frac{a_1}{2} H\right)e_{1} \right\| + \left\| \left(\frac{I}{2} - \frac{a_2}{2} H\right)e_{1} \right\| 
=\frac{1}{2} G_H(a_1) + \frac{1}{2} G_H(a_2)\,,
\end{align*}
which implies that $G_H(a)$ is a convex function. On the other hand, the function $g(x) = x^s$ is convex for $s\geq 1$ on the positive real axis, therefore the composition $g(G_H(a))$ is also convex for any $H$ fixed. Since the expectation of random convex functions is also convex, we conclude that $h(s)$ is also convex. 
\end{proof}

The next result is used in the proof
of Theorem~\ref{thm:moment} to bound
the moments of the iterates.

\begin{lemma}\label{lem:moment:ineq}
(i) Given $0<p\leq 1$, for any $x,y\geq 0$,
\begin{equation}
(x+y)^{p}\leq x^{p}+y^{p}.
\end{equation}

(ii) Given $p>1$, for any $x,y\geq 0$,
and any $\epsilon>0$, 
\begin{equation}
(x+y)^{p}\leq(1+\epsilon)x^{p}+\frac{(1+\epsilon)^{\frac{p}{p-1}}-(1+\epsilon)}{\left((1+\epsilon)^{\frac{1}{p-1}}-1\right)^{p}}y^{p}.
\end{equation}
\end{lemma}

\begin{proof}
(i) If $y=0$, then $(x+y)^{p}\leq x^{p}+y^{p}$
trivially holds. If $y>0$, it is equivalent to show
that
\begin{equation}
\left(\frac{x}{y}+1\right)^{p}\leq\left(\frac{x}{y}\right)^{p}+1,
\end{equation}
which is equivalent to show that
\begin{equation}
(x+1)^{p}\leq x^{p}+1,\qquad\text{for any $x\geq 0$}.
\end{equation}
Let $F(x):=(x+1)^{p}-x^{p}-1$
and $F(0)=0$ and $F'(x)=p(x+1)^{p-1}-px^{p-1}\leq 0$
since $p\leq 1$, which shows that
$F(x)\leq 0$ for every $x\geq 0$. 

(ii) If $y=0$, then the inequality trivially holds. If $y>0$, by doing the transform $x\mapsto x/y$
and $y\mapsto 1$, it is equivalent to show
that for any $x\geq 0$,
\begin{equation}
(1+x)^{p}\leq(1+\epsilon)x^{p}+\frac{(1+\epsilon)^{\frac{p}{p-1}}-(1+\epsilon)}{\left((1+\epsilon)^{\frac{1}{p-1}}-1\right)^{p}}.
\end{equation}
To show this, we define
\begin{equation}
F(x):=(1+x)^{p}-(1+\epsilon)x^{p},
\qquad x\geq 0.
\end{equation}
Then $F'(x)=p(1+x)^{p-1}-p(1+\epsilon)x^{p-1}$
so that $F'(x)\geq 0$ if 
$x\leq((1+\epsilon)^{\frac{1}{p-1}}-1)^{-1}$,
and $F'(x)\leq 0$ if
$x\geq((1+\epsilon)^{\frac{1}{p-1}}-1)^{-1}$.
Thus, 
\begin{equation}
\max_{x\geq 0}F(x)
=F\left(\frac{1}{(1+\epsilon)^{\frac{1}{p-1}}-1}\right)
=\frac{(1+\epsilon)^{\frac{p}{p-1}}-(1+\epsilon)}{\left((1+\epsilon)^{\frac{1}{p-1}}-1\right)^{p}}.
\end{equation}
The proof is complete.
\end{proof}






\end{document}